\documentclass[a4paper,11pt]{article}
\usepackage[T1]{fontenc}
\usepackage[latin1]{inputenc}
\usepackage{color}
\usepackage{xcolor}
\usepackage{array}
\usepackage{hyperref}
\usepackage[all]{xy}
\usepackage{euscript}
\usepackage{mathrsfs}
\usepackage{graphicx}
\usepackage{xspace}
\usepackage{multirow}
\usepackage{nicefrac}
\usepackage{textcomp}
\usepackage{url}
\usepackage{stmaryrd}
\usepackage{amsthm}
\usepackage{geometry}
\usepackage{mathtools}
\usepackage{epsfig}
\usepackage{txfonts}
\usepackage{young}
\usepackage[vcentermath]{youngtab}

\definecolor{purple}{rgb}{0.8,0.12,0.8}
\definecolor{orange}{rgb}{1.0,0.7,0.0}
\definecolor{pink}{rgb}{1,0.5,0.8}
\definecolor{blackg}{rgb}{0.1,0.25,0.1}
\definecolor{ForestGreen}{cmyk}{0.91,0,0.88,0.42}
\definecolor{Turquoise}{cmyk}{0.85,0,0.20,0}

\newcommand{\cF}{\mathcal{F}}

\newcommand{\cO}{\mathcal{O}}

\newcommand{\cU}{\mathcal{U}}

\newcommand{\br}{\mathbf{r}}
\newcommand{\bs}{\mathbf{s}}

\newcommand{\fB}{\mathfrak{B}}

\newcommand{\fS}{\mathfrak{S}}

\newcommand{\fb}{\mathfrak{b}}

\newcommand{\fl}{\mathfrak{l}}

\newcommand{\fs}{\mathfrak{s}}

\newcommand{\sC}{\mathscr{C}}
\newcommand{\sD}{\mathscr{D}}

\newcommand{\sH}{\mathscr{H}}

\newcommand{\sP}{\mathscr{P}}
\newcommand{\sQ}{\mathscr{Q}}

\newcommand{\sS}{\mathscr{S}}

\newcommand{\Z}{\mathbb{Z}}

\newcommand{\N}{\mathbb{N}}

\newcommand{\C}{\mathbb{C}}

\newcommand{\al}{\alpha}
\newcommand{\be}{\beta}
\newcommand{\si}{\sigma}
\newcommand{\la}{\lambda}
\newcommand{\ga}{\gamma}

\newcommand{\ze}{\zeta}

\newcommand{\de}{\delta}

\newcommand{\La}{\Lambda}
\newcommand{\Ga}{\Gamma}

\newcommand{\De}{\Delta}

\newcommand{\bla}{\boldsymbol{\la}}
\newcommand{\bnu}{\boldsymbol{\nu}}
\newcommand{\bmu}{\boldsymbol{\mu}}

\newcommand{\ta}{\tilde{a}}

\newcommand{\tc}{\tilde{c}}

\newcommand{\te}{\tilde{e}}
\newcommand{\tf}{\tilde{f}}

\newcommand{\tga}{\tilde{\gamma}}

\newcommand{\tbs}{\tilde{\bs}}

\newcommand{\tbla}{\tilde{\boldsymbol{\la}}}

\newcommand{\hga}{\hat{\ga}}

\newcommand{\ra}{\rightarrow}
\newcommand{\lra}{\longrightarrow}

\newcommand{\eq}{\Leftrightarrow}

\newcommand{\mand}{\quad\text{and}\quad}

\newcommand{\Uqe}{\mathcal{U}_q(\mathfrak{sl}_e)}

\newcommand{\Ueprime}{\mathcal{U}_q' (\widehat{\mathfrak{sl}_e})}
\newcommand{\Uinf}{\mathcal{U}_q (\mathfrak{sl}_\infty)}
\newcommand{\bemptyset}{\boldsymbol{\emptyset}}
\newcommand{\ds}{\displaystyle}

\newcommand{\llb}{\llbracket}
\newcommand{\rrb}{\rrbracket}
\newcommand{\wt}{\text{\rm wt}}
\newcommand{\Id}{\text{\rm Id}}

\newcommand{\h}{\text{\rm h}}
\newcommand{\res}{\text{\rm res}}
\newcommand{\cont}{\text{\rm cont}}

\newcommand{\rs}{\textbf{\rm RS}}
\newcommand{\reading}{\text{\rm read}}
\newcommand{\bladot}{\overset{\bullet}{\bla}}
\newcommand{\bmudot}{\overset{\bullet}{\bmu}}

\newcommand{\dix}{10}
\newcommand{\onze}{11}
\newcommand{\douze}{12}
\newcommand{\treize}{13}

\newcommand{\bdeux}{\boldsymbol{2}}
\newcommand{\btrois}{\boldsymbol{3}}
\newcommand{\bquatre}{\boldsymbol{4}}
\newcommand{\bcinq}{\boldsymbol{5}}
\newcommand{\bsix}{\boldsymbol{6}}
\newcommand{\bsept}{\boldsymbol{7}}
\newcommand{\bhuit}{\boldsymbol{8}}

\theoremstyle{plain}
\newtheorem{thm}{Theorem}[section]
\newtheorem{lem}[thm]{Lemma}

\newtheorem{cor}[thm]{Corollary}
\newtheorem{prop}[thm]{Proposition}
\newtheorem{propdef}[thm]{Proposition-Definition}

\newtheorem{prope}[thm]{Property}

\theoremstyle{definition}
\newtheorem{defi}[thm]{Definition}
\newtheorem{notation}[thm]{Notation}

\theoremstyle{remark}
\newtheorem{exa}[thm]{Example}
\newtheorem{rem}[thm]{Remark}

\makeindex

\title{Crystal isomorphisms in Fock spaces and Schensted correspondence in affine type $A$}
\author{Thomas Gerber 
\footnote{\textsc{Lehrstuhl D f\"ur Mathematik, RWTH.}
Pontdriesch 14, 52062 Aachen, Germany. 
E-mail address : \url{gerber@math.rwth-aachen.de} 
}
}

\begin{document}
 
\maketitle

\renewcommand{\labelitemi}{$-$}

\abstract{
We are interested in the structure of the crystal graph of level $l$ Fock spaces representations of $\Ueprime$.
Since the work of Shan \cite{Shan2008}, we know that this graph encodes the modular branching rule for a corresponding cyclotomic rational Cherednik algebra.
Besides, it appears to be closely related to the Harish-Chandra branching graph for the appropriate finite unitary group, according to \cite{GerberHissJacon2014}.
In this paper, we make explicit a particular isomorphism between connected components of the crystal graphs of Fock spaces. 
This so-called "canonical" crystal isomorphism turns out to be expressible only in terms of:
\begin{itemize}
 \item Schensted's classic bumping procedure,
 \item the cyclage isomorphism defined in \cite{JaconLecouvey2010},
 \item a new crystal isomorphism, easy to describe, acting on cylindric multipartitions.
\end{itemize}
We explain how this can be seen as an analogue of the bumping algorithm for affine type $A$.
Moreover, it yields a combinatorial characterisation of the vertices of any connected component of the crystal of the Fock space.
}

\renewcommand{\labelitemi}{$\bullet$}

\section{Introduction}

To understand the representations of the quantum algebras of affine type $A$, 
Jimbo, Misra, Miwa and Okado introduced in \cite{JMMO1991} Fock spaces of arbitrary level $l$.
These are vectors spaces over $\C(q)$ depending on a parameter $\bs\in\Z^l$, with basis the set of all $l$-partitions. 
We denote by $\cF_\bs$ such a space.
They made explicit an action of the quantum algebra $\Ueprime$ which endows $\cF_\bs$
with the structure of an integrable $\Ueprime$-module.
In addition, there is a natural subrepresentation $V(\bs)$ of $\cF_\bs$, the one generated by the empty $l$-partition,
that gives a concrete realisation of any abstract highest weight $\Ueprime$-module, provided $\bs$ is suitably chosen.

According to the works of Kashiwara \cite{Kashiwara1991}, these abstract modules have a nice underlying combinatorial structure: their crystal.
In particular, we can define the so-called crystal operators, crystal graph, crystal basis, and global basis for such modules.
The combinatorial nature of Fock spaces provides a convenient framework for the study of the crystal structure
of the highest weight $\Ueprime$-modules.
Actually, the whole space $\cF_\bs$ has a crystal structure, and the features of the highest weight $\Ueprime$-representations (most interestingly the global basis)
also exist more generally for the whole module $\cF_\bs$. 
This was the point of Uglov's work \cite{Uglov1999} (see also \cite{Yvonne2005} or \cite{Leclerc2008}).

\medskip

As a matter of fact, the study of crystals of Fock spaces lies at the core of several algebraic problems.
First of all, there exists an intimate relation with the modular representation theory of complex reflection groups.
The first achievement in this direction was Ariki's proof in \cite{Ariki1996} of the LLT conjecture \cite{LLT1996},
enabling the computation of the decomposition matrices for Hecke algebras of type $G(l,1,n)$ (also known as Ariki-Koike algebras) via
the matrices of the canonical basis of level $l$ Fock spaces.
His categorification arguments also proved that the crystal graph of $V(\bs)$ describes the branching rule for the Ariki-Koike algebra, see \cite{Ariki2007}.
From then, many authors have highlighted the links between the two theories in papers such as \cite{Grojnowski1999}, \cite{Uglov1999} or \cite{BrundanKleshchev2001}.
For instance, the matrix of Uglov's canonical basis of $\cF_\bs$ is seen as the decomposition matrix of a cyclotomic $q$-Schur algebra, see \cite{VaragnoloVasserot1999}.

Actually, the crystal structure on Fock spaces has even deeper interpretations.
According to \cite{Shan2008}, the category $\cO$ for representations of the rational Cherednik algebras $\sH$ associated to $G(l,1,n)$, $n\geq0$,
also has a crystal structure, given, again, by some $i$-restriction and $i$-induction functors,
and which is isomorphic to the crystal of the whole $\cF_\bs$ (for the appropriate $\bs$).
In fact, Losev has proved in \cite[Theorem 2.1]{Losev2013} that the branching graph arising from Shan's induction and restriction functors 
precisely coincides with Uglov's version of the crystal graph of the Fock space.
Accordingly, the main result of this paper, namely Theorem \ref{thmisophi}, is expected to have an interpretation in terms of representation
theory of rational Cherednik algebras.
Note also that understanding the crystal structure on $\cF_\bs$ helps determining the finite-dimensional modules for these Cherednik algebras,
as explained in the papers by Shan and Vasserot \cite[Section 6]{ShanVasserot2012} and by Gordon and Losev \cite[Section 6.9.3]{GordonLosev2014}.

Finally, let us mention that according to the recent observations exposed in \cite{GerberHissJacon2014}, 
the information encoded by the crystal graph of an appropriate level $2$ Fock space conjecturally sheds some light on another completely different problem,
namely understanding the distribution of the unipotent modules for finite unitary groups into Harish-Chandra series.
More precisely, the crystal of this Fock space should represent the Harish-Chandra branching graph for the associated unitary group.
In particular, \cite[Conjecture 5.7]{GerberHissJacon2014} implies the existence of two different labelings for the same Harish-Chandra series.
The crystal isomorphism constructed in the present paper (Corollary \ref{corcci}) then permits to check 
that these two labelings are indeed compatible, this is \cite[Theorem 7.8]{GerberHissJacon2014}.

\medskip

This article is concerned about the more specific study of crystal graphs within Fock spaces. 
There are several connected components in such a graph, each of which is entirely determined 
(and hence parametrised) by its so-called highest weight vertex.
In particular, the connected component with highest weight vertex the empty $l$-partition is
the crystal graph of the highest weight representation $V(\bs)$.
What we call a crystal isomorphism within Fock spaces is essentially a graph isomorphism. 
We demand that it maps the highest weight vertex of some connected component to
the highest weight vertex of some other connected component (a priori of a different Fock space),
and that it intertwines the structure of oriented colored graph.
Because of the ubiquitous combinatorics in Fock spaces, we expect these crystal isomorphisms
to have an nice combinatorial description.

In their papers \cite{JaconLecouvey2010} and \cite{JaconLecouvey2012}, Jacon and Lecouvey have gathered partial information
about such mappings.
Indeed, in \cite{JaconLecouvey2010}, in order to generalise the results of \cite{Jacon2007} for $l=2$,
they described the crystal isomorphisms mapping the crystal graph of $V(\bs)$ to that of $V(\br)$,
when $\bs$ and $\br$ are in the same orbit under the action of $\widehat{\fS_l}$.
In \cite{JaconLecouvey2012}, they explained how an arbitrary crystal isomorphism should act on a highest weight vertex.

The point of this article is to solve the problem entirely, by determining crystal isomorphisms in Fock spaces in a more general setting.
By Proposition \ref{defcanonicaliso}, each connected component ot the crystal graph of
$\cF_\bs$ is isomorphic to the crystal of some $V(\br)$, 
where $\br$ belongs to a particular fundamental domain for the action of $\widehat{\fS_l}$ for which a simple combinatorial description of this crystal is known
(its vertices are the "FLOTW" $l$-partitions defined in \ref{defflotw}).
The question is then to determine the associated crystal isomorphism $\Phi$.
Precisely, we aim to express explicitely and combinatorially the action of $\Phi$ on any $\bla\in\cF_\bs$.
We call this particular mapping the \textit{canonical $\Ueprime$-crystal isomorphism} (Definition \ref{defcanonicaliso}).
Drawing a parallel with the non-affine case (representations of $\Uqe$, or $\Uinf$),
we can then regard the construction $\bla \mapsto \Phi(\bla)$ as an affine analogue of Schensted's bumping algorithm.
Indeed, in the case of regular type $A$, it is known (Theorem \ref{thmrs}) 
that the canonical crystal isomorphism is exactly the bumping procedure.
Roughly speaking, we replace the bumping algorithm on semistandard tableaux by a rectification 
procedure on symbols (also called keys or tabloids in the literature) yielding, in turn, semistandard symbols, cylindric symbols, and FLOTW symbols.

\medskip

The paper is organized as follows.
Section \ref{problem} contains the basic notations and definitions needed to handle the combinatorics of crystals in Fock spaces, 
as well as a quick review on the theory of $\Ueprime$-representations, and more particularly the Fock space representations.
We also formally express the problem we wish to solve, which translates to finding the canonical $\Ueprime$-crystal isomorphism.

In Section \ref{e=inf}, we first roughly explain how the limit case $e\ra\infty$ gives another structure on Fock spaces,
namely that of a $\Uinf$-module. 
We then recall some classic results about $\Uinf$-crystals (and their relation to Schensted's bumping procedure) 
and explain how this solves our problem when Fock spaces are considered as $\Uinf$-modules.
In this perspective, the  canonical $\Ueprime$-crystal isomorphism we construct in the last section 
can be seen as an affine version of the bumping algorithm.

Then, we show in Section \ref{reduction} that both structures are compatible. More precisely, any $\Uinf$-crystal isomorphism 
is also a  $\Ueprime$-crystal isomorphism (Proposition \ref{propcompat}).
Using this, we determine a way to restrict ourselves only to the study of so-called cylindric multipartitions.
This requires the results about $\Ueprime$-crystal isomorphisms obtained in \cite{JaconLecouvey2010}.
In particular, there exists a particular crystal isomorphism, which consists of "cycling" multipartitions,
and which is typical of the affine case.
Note that this cyclage procedure is constructed in a similar way as the cyclage procedure on tableaux
defined by Lascoux and Sch\"utzenberger in \cite[Chapter 5]{Lothaire2002} in the non-affine case, see also \cite{Shimozono2001}.

Section \ref{cylindric} treats the case of cylindric multipartitions.
The key ingredient is Theorem \ref{thmisopsi}, and the expected canonical crystal isomorphism is described in Theorem \ref{thmisophi}.
Together with Section \ref{reduction}, this eventually enables the determination of the canonical crystal isomorphism $\Phi$ in full generality.
We also explain in which sense we regard our isomorphism $\Phi$ as an affine analogue of Schensted's bumping procedure.

Finally, Section \ref{hw} is an application of these results. We deduce a direct (i.e. pathfinding-free) algorithmic
characterisation of the vertices of any connected component of the crystal graph of the Fock space.
Note that this requires the invertibility of the map $\Phi$. This is achieved by adding some "recording data" to the construction of $\Phi$,
and gives an affine analogue of the whole (one-to-one) Robinson-Schensted-Knuth correspondence.

Appendix \ref{appendix} contains the proofs of three technical key lemmas stated in Section \ref{cylindric}.

\section{Formulation of the problem}

\label{problem}

\subsection{Generalities}

\label{generalities}

We recall the usual notations and definitions about the combinatorics of multipartitions. 

Fix $n\in\N$ and $l\in\Z_{>0}$.
A \textit{partition} $\la$ of $n$ is a sequence $(\la_1,\la_2,\dots)$ such that $\la_a\geq \la_{a+1}$ for all $a$, and $\sum_a \la_a =n$.
The integer $n$ is then called the \textit{rank} of $\la$ and denoted by $|\la|$. Each $\la_a$ is called a \textit{part} of $\la$.
For convenience, we often consider that a partition has infinitely many parts equal to zero.
The set of partitions of $n$ is denoted by $\Pi(n)$.
We also set $h(\bla)=\max_{\la_a\neq 0} a$, and call $h(\bla)$ the \textit{height} of $\bla$.
Out of simplicity, we use the multiplicative notation for a partition $\la$. For instance, $(4,2,2,1,0,\dots)=(4.2^2.1)$.
A \textit{$l$-partition} (or simply \textit{multipartition}) $\bla$ of $n$ is a $l$-tuple 
of partitions $\la^c$, for $c\in\llb 1, l\rrb$ such that $\sum_{c=1}^l |\la^c| =n$.
The integer $n$ is again called the \textit{rank} of $\bla$ and denoted by $|\bla|$;
and we write $\Pi_l(n)$ for the set of $l$-partitions of $n$.
Moreover, we denote by $\emptyset$ the partition $(0,0,\dots)$ and by $\bemptyset$ the multipartition $(\emptyset,\dots,\emptyset)$.

A multipartition $\bla$ is often identified with its Young diagram $[\bla]$:

$$[\bla]:=\Big\{ \ \ (a,b,c) \ \ ;  \ \ a\geq 1 , c \in\llb 1, l\rrb, b\in\llb 1, \la_a^c\rrb \ \ \Big\}$$

In turn, the Young diagram of $\bla$ is itself depicted by a $l$-tuple of array where the $c$-th array is the superposition of $\la_a^c$ boxes, $a\geq1$.
For example, $[(1, 2^2,\emptyset)]= \left( \; \tiny \yng(1) \; , \; \yng(2,2) \; , \; \emptyset  \right)$.
Each box in this diagram is then labeled by an element $(a,b,c)$ of $[\bla]$. These elements are called the \textit{nodes} of $\bla$.
A node $\ga$ of $\bla$ is said to be \textit{removable} (or a node \textit{of type $R$}) if $[\bla]\backslash \{ \ga \}$ is still the Young diagram of some $l$-partition $\bmu$.
In this case, we also call $\ga$ an \textit{addable} node (or a node \textit{of type $A$}) of $\bmu$.

A \textit{$l$-charge} (or simply \textit{multicharge}) is a $l$-tuple $\bs=(s_1,\dots s_l)$ of integers.
A \textit{charged multipartition} is a formal pair denoted $|\bla,\bs\rangle$, where $\bla$ is a $l$-partition and $\bs$ is a $l$-charge.
We can then define the \textit{content} of a node $\ga=(a,b,c)$ of $\bla$ charged by $\bs$ as follows:

$$\cont_{\bla}(\ga)= b - a + s_c .$$

Also, given $e\in\Z_{>1}$, we define the \textit{residue} of $\ga$ by $$\res_{\bla} (\ga) = \cont_{\bla}(\ga) \mod e.$$

Let $i\in\llb 0, e-1 \rrb$. A node $\ga$ of $\bla$ is called an \textit{$i$-node} if $\res_{\bla}(\ga) = i $.
If $\ga=(a,b,c)$ is a node of $\bla$, denote by $\ga^+$ the node $(a,b+1,c)$, that is
the node located on the right of $\ga$. Similarly, denote by $\ga^-$ the node $(a,b-1,c)$, the node located on the left of $\ga$.
Of course, $\mu=\ga^+ \eq \ga=\mu^-$. Also, if $\ga$ is an $i$-node, $\ga^+$ is an $(i+1)$-node.

We can represent each charged multipartition $|\bla,\bs \rangle$ by its Young diagram whose boxes are filled by the associated contents.
\newcommand{\moinsun}{$-$1}
\newcommand{\moinsdeux}{$-$2}
For instance, $$|(4.1, 2^2, 3 ),(0,3,-2) \rangle = \left( \; \tiny \young(0123,\moinsun) \; , \; \young(34,23) \; , \; \young(\moinsdeux\moinsun0)  \right).$$

\medskip

There is an equivalent way to represent charged multipartitions by another combinatorial object, namely by the so-called \textit{symbols}.
Let us recall their definition, following \cite{GeckJacon2011}.
Having fixed a multipartition $\bla$ and a multicharge $\bs$, take $p\geq \max_c(1-s_c+h(\la^c))$.
We can first define $$\fb_a^c (\bla) := \la_a^c - a + p + s_c,$$ for $1\leq c \leq l$ and $1\leq a \leq p+s_c$.
We then set $\fB_\bs^c(\bla) = ( \fb_{p+s_c}^c(\bla) , \dots, \fb_1^c(\bla) ),$ and the $\bs$-symbol of $\bla$ of size $p$ is the following $l$-tuple:
$$\fB_\bs(\bla) = (\fB_\bs^1(\bla),\dots,\fB_\bs^l (\bla) ).$$
It is pictured in an array whose $c$-th row, numbered from bottom to top, is $\fB_\bs^c(\bla)$.

\begin{exa}\label{exasymbol}
Take $\bla=(4.1, 2^2, 3 )$ and $\bs=(0,3,-2)$.
Then we can take $p=4$, and we get $\fB_\bs^1(\bla)= (0,1,3,7)$,  $\fB_\bs^2(\bla)=(0,1,2,3,4,7,8)$ and $\fB_\bs^3(\bla)=(0,4)$, which is represented by 
$$\fB_\bs(\bla)=\begin{pmatrix}0 & 4 & & & & &\\ 
                 0 & 1  & 2& 3 & 4 & 7 & 8 \\
                 0 & 1 & 3 & 7 & & &
                \end{pmatrix} .$$
\end{exa}

Note that one recovers the elements $\fb_a^c(\bla)$ encoding non-zero parts by translating by $p$ the contents of the rightmost nodes in the Young diagram.
In fact, it is easy to see that both objects encode exactly the same data.
Throughout this paper, we often switch between the classic approach of "Young diagrams with contents" on the one hand, and the use of symbols on the other hand.

\begin{defi}\label{defsemistandard} Let $\bs\in\Z^l$ and $\bla\in\cF_\bs$.
The symbol $\fB_\bs(\bla)$ is called \textit{semistandard} if the three following conditions are satisfied:
\begin{itemize}
\item $s_1\leq s_2 \leq \dots \leq s_l$,
\item the entries in each column of $\fB_\bs(\bla)$ are non-decreasing,
\item the entries in each row of $\fB_\bs(\bla)$ are increasing.
\end{itemize}
\end{defi}

\medskip

For $e\in\Z_{>1}$, we denote $$\sS_e=\{ \bs \in \Z^l \, | \, 0\leq s_c - s_{c'} < e \text{ for } c'<c \}.$$

We define a family of particular multipartitions, the \textit{FLOTW multipartitions} (for Foda, Leclerc, Okado, Thibon, Welsh \cite{FLOTW1999}).
We will see in the Section \ref{equivalentmp} why this is an object of interest in this paper.

\begin{defi}\label{defflotw}
Let $\bs\in\sS_e$.
A charged multipartition $|\bla,\bs\rangle$ is called FLOTW
\footnote{In the paper \cite{FLOTW1999}, the authors use the original terminology of \cite{JMMO1991},
namely the FLOTW multipartitions are called \textit{highest-lift} multipartitions. 
In \cite{JMMO1991}, the combinatorics are expressed in terms of \textit{paths} and \textit{patterns},
and cylindric multipartitions correspond to the so-called \textit{normalized} patterns.
}
if:
\begin{itemize}
 \item For all $1\leq c \leq l-1$, $\la_a^c\geq \la_{a+s_{c+1}-s_c}^{c+1}$ for all $a\geq 1$; and $\la_a^l\geq \la_{a+e+s_1-s_l}$ for all $a\geq 1$.
\item For all $\al>0$, the residues of the nodes $(a,\al,c)$ with $\la_a^c=\al$ (i.e. the rightmost nodes of the rows of length $\al$ of $\bla$)
do not cover $\llb 0, e-1 \rrb$.
\end{itemize}
If $|\bla,\bs\rangle$ only verifies the first condition, we say that it is \textit{cylindric}.
\end{defi}

\begin{rem}
In particular, the symbol of a FLOTW $l$-partition is semistandard.
\end{rem}

\medskip

Multipartitions are the natural objects used to understand combinatorially the representations of the quantum algebras of affine type $A$.
In the following section, we recall some definitions and facts about these algebras, their representations, 
before focusing in Section \ref{equivalentmp} on the crystal structure they are endowed with.

\subsection{Fock spaces representations of quantum algebras}

\label{fockspace}

In the sequel, we fix $e\in\Z_{>1}$.

For the purpose of this article, it is not necessary to review the theory of quantum algebras.
The reader is invited to refer to \cite{Ariki2002} and \cite{HongKang2002} for detailed definitions and properties.
However, we recall quickly below the theory of highest weight integrable representations of $\Ueprime$. 

Essentially, the quantum algebra $\Ueprime$ is a $\C(q)$-algebra which is a one-parameter deformation 
of the universal enveloping algebra of the affine Kac-Moody algebra $\widehat{\fs\fl_e}$,
whose underlying Weyl group is the affine group of type $A_{e-1}$.
It is defined by generators, denoted by $e_i, f_i, t_i^{\pm 1}$ for $i\in\llb 0, e-1\rrb$, and some relations.
Besides, it is possible to endow it with a coproduct, which makes it a Hopf algebra.

\medskip

Denote $\La_i$, $i\in\llb 0, e-1 \rrb$ and $\de$ the fundamental weights for $\widehat{\fs\fl_e}$.
Recall that the simple roots are then given by $\al_i:=-\La_{i-1 \text{mod} e} + 2\La_{i} - \La_{i+1\text{mod} e}$ for $0\leq i \leq e-1$.
According to \cite[Chapter 6]{Ariki2002}, to each dominant integral weight $\La$ is associated a particular $\Ueprime$-module $V(\La)$,
called the \textit{the highest weight module of highest weight $\La$},
verifying the following property:

\begin{prope}\label{ueprime}
We have $V(\La^{(1)}) \simeq V(\La^{(2)})$ as $\Ueprime$-modules if and only if $\La^{(1)}-\La^{(2)}\in\Z\de$.
\end{prope}

Moreover, provided we work in the category of so-called "integrable" $\Ueprime$-modules,
it is known that:
\begin{enumerate}
\item For all dominant integral weight $\La$, the module $V(\La)$ is irreducible,
\item Each irreducible $\Ueprime$-module is isomorphic to some $V(\La)$,
\item Each $\Ueprime$-module is semisimple.
\end{enumerate}

%Let $\La=\sum_{0\leq i <e} a_i\La_i + d\de$ be an integral weight.
%In the sequel, if $v$ is an element of a $\Ueprime$-module, we denote $\wt(v)$ the weight of $v$.
%A \textit{highest weight vector of weight $\La$}  is a vector $v$  which decomposes as $v=\sum_k \al_k m_k$ where:
%
%- $\al_k \in \C(q)$ for all $k$
%
%- the $m_k$'s are elements of a $\Ueprime$-module, verifying $t_i.m=q^{a_i}m$ for $0\leq i <e$.
%
%- $e_i.v = 0$.
%
%A $\Ueprime$-module $M$ is then said to be \textit{a highest weight module with highest weight $\La$} if there is a highest weight vector $v_\La$
%of weight $\La$ such that $M=\Ueprime.v_\La$.
%In fact, for each $\La$, there exists a unique simple highest weight $\Ueprime$-module with highest weight $\La$, denoted by $V(\La)$.
%
%Moreover, the following property holds \cite{Ariki2002}:

%
%
%Actually, there is a nice category of $\Ueprime$-representations, the \textit{integrable} representations, see \cite[Chapter 6]{Ariki2002},
%which has the following properties:
%\begin{itemize}
%\item If $\La$ is a dominant integral weight (i.e. the $a_i$'s are non-negative), then $V(\La)$ is irreducible and integrable.
%\item Each irreducible and integrable highest weight $\Ueprime$-module is isomorphic to some $V(\La)$ with $\La$ a dominant integral weight.
%\item This category is semisimple.
%\end{itemize}

\medskip

We now construct a practical representation of $\Ueprime$, the \textit{Fock space representation},
which turns out to be integrable and has nice combinatorial properties. 
In particular, this will yield a realisation of the highest weight $\Ueprime$-modules $V(\La)$.
In this perspective, we fix $l\in\Z_{>0}$.

\begin{defi}\label{deffockspace} Let $\bs=(s_1,\dots,s_l)\in\Z^l$.
The Fock space associated to $\bs$ is the following $\C(q)$-vector space: $$\cF_\bs := \bigoplus_{n\in\N} \bigoplus_{\la\in\Pi_l(n)} \C(q) |\bla,\bs \rangle.$$
\end{defi}

This vector space is made into an integrable $\Ueprime$-module via the action of the generators of $\Ueprime$ detailed in \cite[Proposition 3.5]{JMMO1991}.
We do not recall here what theses actions are, since it is not important in our context.
However, with this definition of the action, it is easy to see that the element $|\bemptyset,\bs\rangle$ is a highest weight vector. Denote by $\La(\bs)$ its weight.
By \cite[Section 4.2]{Uglov1999} or \cite[Proposition 3.7]{Yvonne2005}, we now that for all $|\bla,\bs\rangle$ in $\cF_\bs$,
\begin{equation}\label{wt}
 \wt(|\bla,\bs\rangle)=\sum_{c=1}^l\La_{s_c \text{ mod} e} -\sum_{i=0}^{e-1}M_i(\bla,\bs)\al_i - \De(\bs)\de,
\end{equation}

where $M_i(\bla,\bs)$ is the number of $i$-nodes in $\bla$, 
and where $\De(\bs)$ is a coefficient depending only on $\bs$ and $e$.
In particular, we see that
\begin{equation}\label{weight} \La(\bs)=\wt(|\bemptyset,\bs\rangle) = \sum_{c=1}^l\La_{s_c \text{ mod} e} - \De(\bs)\de. \end{equation}
Consider the module $$V(\bs):= \Ueprime.|\bemptyset,\bs\rangle.$$
It is an irreducible integrable highest weight $\Ueprime$-module with highest weight $\La(\bs)$. 
Hence, by the previous properties of $\Ueprime$-representations, we know that this module $V(\bs)$ is isomorphic to the "abstract" $\Ueprime$-module $V(\La(\bs))$.

\medskip

Using Property \ref{ueprime} and Relation (\ref{weight}), we have a $\Ueprime$-module isomorphism between $V(\bs)$ and $V(\br)$
whenever $\br$ is equal to $\bs$ up to a permutation of its components and translations by a multiple of $e$.
In fact, when two multicharges are related in such a way, we regard them as being in the same orbit under some group action.
This is why we now introduce the \textit{extended affine symmetric group} $\widehat{\fS}_l$.
It is the group with the following presentation:
\begin{itemize}
 \item Generators: $\si_i$, $i\in\llb 1, l-1\rrb$ and $y_i$, $i\in\llb 1, l\rrb$.
 \item Relations : \begin{itemize}
                       \item $\si_i^2=1$ for $i\in\llb 1, l-1\rrb$,
                       \item $\si_i\si_{i+1}\si=\si_{i+1}\si_i\si_{i+1}$  for $i\in\llb 1, l-2\rrb$,
                       \item $\si_i\si_j=\si_j\si_i$ if $i-j\neq 1 \mod l$,
                       \item $y_iy_j=y_jy_i$ for $i,j\in\llb 1, l\rrb$,
                       \item $\si_i y_j=y_j\si_i$ for $i\in\llb 1, l-1\rrb$ and $j\in\llb 1, l\rrb$ such that $j\neq i, i+1 \mod l$,
                       \item $\si_i y_i\si_i=y_{i+1}$ for $i\in\llb 1, l-1\rrb$.
                    \end{itemize}

\end{itemize}

It can be regarded as the semi-direct product $\Z^l\rtimes\fS_l$,
by considering that the elements $y_i$ form the standard basis of $\Z^l$, 
and that the elements $\si_i$ are the usual generators of $\fS_l$ (i.e. the transpositions $(i , i+1)$).

Now, there is a natural action of $\widehat{\fS}_l$ on $\Z^l$ as follows. Let $\bs=(s_1,\dots s_l) \in\Z^l$. We set
\begin{itemize}
 \item $\si_i . \bs = (s_1, \dots, s_{i-1},s_{i+1},s_i, \dots, s_l)$ for all $i\in\llb 1, l-1\rrb$, and
 \item $y_i .\bs=(s_1,\dots,s_{i-1},s_i+e,s_{i+1},\dots, s_l)$ for all $i\in\llb 1, l\rrb$.
\end{itemize}

It is easy to see that a fundamental domain for this action is given by
$$\sD_e=\{ \bs \in \Z^l \, | \, 0\leq s_1 \leq \dots \leq s_l < e \}.$$

\begin{rem}
Note that 
$$\sD_e \subset \sS_e.$$
It is important not to be confused between $\sD_e$ and the set  $\sS_e$ defined just before Definition \ref{defflotw}.
Indeed, it is sufficient to work in $\sS_e$ to have the explicit combinatorial caracterisation of the crystal of $V(\bs)$,
but it is necessary to work with $\sD_e$ if we expect some unicity properties (as in upcoming Proposition \ref{defcanonicaliso}).
\end{rem}

With these notations, it is clear that $V(\bs)\simeq V(\br)$ as $\Ueprime$-modules if and only if $\br$ and $\bs$ are
in the same orbit under the action of $\widehat{\fS_l}$.
Actually, there is a tight connection between these highest weight $\Ueprime$-representations $V(\bs)$ and the modular
representations of Ariki-Koike algebras (the Hecke algebras of the complex reflection groups $G(l,1,n)$),
which are defined using the parameters $e$ and $\bs$, and which are invariant when $\bs$ varies inside a given orbit under the action of $\widehat{\fS_l}$.
The main link is known as Ariki's theorem \cite[Theorem 12.5]{Ariki2002},\cite{Ariki1996}, which is a proof of the LLT conjecture \cite{LLT1996}.
A good review on this subject can also be found in the book of Geck and Jacon \cite[Chapters 5 and 6]{GeckJacon2011}.

\medskip

From now on, we allow ourselves to write $\bla$ instead of $|\bla,\bs\rangle$ for an element of a Fock space  if there is no possible confusion on the multicharge.

\subsection{Crystal isomorphisms and equivalent multipartitions}

\label{equivalentmp}

One of the most important features of Fock spaces is the existence of a \textit{crystal}, in the sense of Kashiwara \cite{Kashiwara1991}. 
This yields a nice combinatorial structure, which is in particular encoded in the \textit{crystal graph} of $\cF_\bs$.
Its definition requires an action of the so-called \textit{crystal operators} $\te_i$ and $\tf_i$, for $0\leq i \leq e-1$, on the Fock space $\cF_\bs$.
We do not give their original definition, since the only result we need in the sequel is upcoming Theorem \ref{thmcrystaloperators}.

In order to determine this action, we introduce an order on the set of nodes of a charged multipartition $|\bla,\bs\rangle$.
Let $|\bla,\bs\rangle \in \cF_\bs$ and $\ga=(a,b,c)$ and $\ga'=(a',b',c')$ be two removable or addable $i$-nodes of $\bla$.
We write  $$\ga\prec_\bs \ga' \text{ if }   \left\{      \begin{array}{l}      b-a+s_c<b'-a'+s_{c'} \text{ or }
\\  b-a+s_{c}=b'-a'+s_{c'} \mand c>c'.   \end{array}   \right.$$
Note that this order is needed to define the action of $\Ueprime$ on $\cF_\bs$ in Section \ref{fockspace}.

For all $i\in\llb 0, e-1 \rrb$, the \textit{$i$-word} for $\bla$ is the sequence obtained by writing the addable and removable $i$-nodes of $\bla$
in increasing order with respect to this order $\prec_\bs$, each of them being encoded by a letter $A$ if it is addable, 
and a letter $R$ if it is removable. We denote it by $w_i(\bla)$.
The \textit{reduced $i$-word} $\hat{w}_i(\bla)$ is the word in the letters $A$ and $R$ obtained by deleting recursively all occurences $RA$ in $w_i(\bla)$.
Hence $\hat{w}_i(\bla)=A^\al R^\be$ for some non-negative integers $\al$ and $\be$.
If it exists, the rightmost $A$ (respectively the leftmost $R$) in $\hat{w}_i(\bla)$ encodes a node which is called 
the \textit{good} addable (respectively removable) $i$-node of $\bla$.

\begin{exa}
 If $|\bla,\bs\rangle= \left( \; \tiny \young(0123,\moinsun) \; , \; \young(34,23) \; , \; \young(\moinsdeux\moinsun0)  \right)$, $e=3$,
and $i=0$, then $w_0(\bla)=ARARR$ and therefore  $\hat{w}_0(\bla)=ARR$.
The good addable $0$-node of $\bla$ is thus $(2,1,3)$, and the good removable $0$-node is $(2,2,2)$.
\end{exa}

\medskip

\begin{thm}[\cite{JMMO1991}]\label{thmcrystaloperators}
Let $\bla \in\cF_\bs$.
The crystal operators act as follows:
\begin{itemize}
 \item $\te_i(\bla) = \left\{      \begin{array}{ll}     \bla\backslash\{\ga\} & \text{ \quad if $\ga$ is the good removable $i$-node of $\bla$}
\\  0 & \text{ \quad if $\bla$ has no good removable $i$-node},  \end{array}   \right.$
 \item $\tf_i(\bla) = \left\{      \begin{array}{ll}     \bla\cup\{\ga\} & \text{ \quad if $\ga$ is the good addable $i$-node of $\bla$}
\\  0 & \text{ \quad if $\bla$ has no good addable $i$-node}.  \end{array}   \right.$
\end{itemize}
\end{thm}

\begin{rem}
 Clearly, because of Relation (\ref{wt}), we have $\wt(\te_i(\bla))=\wt(\bla)+\al_i$ and $\wt(\tf_i(\bla))=\wt(\bla)-\al_i$.
\end{rem}

\medskip

We can now define the \textit{crystal graph} of $\bla$.

\begin{defi}\label{defcrystal}
The crystal graph $B(\bla,\bs)$ of $\bla\in\cF_\bs$ is the oriented colored graph with:
\begin{itemize}
 \item vertices : the multipartitions obtained from $\bla$ after applying any combination of the operators $\te_i$ and $\tf_i$.
\item arrows : $\bla \overset{i}{\lra} \bmu$ whenever $\bmu = \tf_i(\bla)$.
\end{itemize}
\end{defi}

Theorem \ref{thmcrystaloperators} says that when $\te_i$ acts non-trivially on $\bla$ (i.e. when $\bla$ has a good removable $i$-node),
then $\te_i$ removes a node in $\bla$.
Hence, any sequence $\dots \te_{i_k}\te_{i_{k-1}} \dots \te_{i_1} (\bla)$ in $B(\bla,\bs)$ has at most $|\bla|$ elements.
In particular, it is finite, and there is a sequence of maximal length $m$ of operators $\te_{i_k}$ such that 
$\te_{i_p}\te_{i_{p-1}} \dots \te_{i_1} (\bla)$ is a multipartition and $\te_{j}\te_{i_{p}} \dots \te_{i_1} (\bla) = 0$ for all $j\in\llb 0, e-1 \rrb$.
Consider the multipartition $$\overset{\bullet}{\bla}:=\te_{i_p}\te_{i_{p-1}} \dots \te_{i_1} (\bla)$$
It is not complicated to show that it does not depend on the maximal sequence of operators $\te_{i_k}$ chosen.
In other words, all sequences give the same multipartition $\overset{\bullet}{\bla}$, which we call the \textit{highest weight vertex} of $B(\bla,\bs)$.
Also, every vertex $\bmu$ in $B(\bla,\bs)$ writes $\bmu=\tf_{i_r}\tf_{i_{r-1}}\dots \tf_{i_1} (\overset{\bullet}{\bla})$ for some $r\in\N$.
Therefore, any crystal graph is entirely determined by its highest weight vertex,
and if we know this highest weight vertex, all other vertices are recursively computable.
It turns out that in one particular case, we know an explicit combinatorial description of the vertices of $B(\bla,\bs)$.

\textbf{Notation:}
When $\overset{\bullet}{\bla}=\bemptyset$, we write $B(\bemptyset,\bs)=:B(\bs)$.

\medskip

\begin{thm}[\cite{FLOTW1999}]\label{thmflotw} Let $\bs\in\sS_e$.
The vertices of $B(\bs)$ are the FLOTW multipartitions.
\end{thm}

Actually, for all $\bs\in\Z^l$, $B(\bs)$ is the crystal graph of the irreducible highest weight $\Ueprime$-module $V(\bs)$ in the sense of Kashiwara \cite{Kashiwara1991}.
Accordingly, the combinatorics of $B(\bla,\bs)$ has some algebraic interpretation.
In particular, the vertices of $B(\bla,\bs)$ (the FLOTW multipartitions if $\bs\in\sS_e$) label both the so-called crystal and global (or canonical) basis of $V(\bs)$.
We also refer to \cite{GeckJacon2011} for details.

\medskip

We can also define a crystal graph for the whole Fock space $\cF_\bs$.
Its vertices are all the $l$-partitions, and the arrows are defined as in Definition \ref{defcrystal}. We denote it by $B(\cF_\bs)$.
It has several connected components, each of them being parametrised by its highest weight vertex. In other terms,
\begin{equation}\label{decompcrystal} B(\cF_\bs) = \bigsqcup_{\overset{\bullet}{\bla}} B(\overset{\bullet}{\bla},\bs),\end{equation}
where the union is taken over all highest weight vertices $\overset{\bullet}{\bla}$, i.e. vertices without any removable node.

We can now introduce the notion of crystal isomorphism.

\begin{defi}\label{crystaliso}Let $\bla\in\cF_\bs$ and $\bmu\in\cF_\br$.
A \textit{$\Ueprime$-crystal isomorphism} is a map $\phi : B(\bladot,\bs) \lra B(\bmudot,\br)$ verifying:
\begin{enumerate}
 \item $\phi(\overset{\bullet}{\bla})=\overset{\bullet}{\bmu}$,
\item $\phi \circ \tf_i = \tf_i \circ \phi$ whenever $\tf_i$ acts non trivially.
\end{enumerate}
\end{defi}

By 2., the image of $B(\bladot,\bs)$ under $\phi$ is the whole crystal $B(\bmudot,\br)$.
In fact, this definition just says that $\phi$ intertwines the graph structures of $B(\bladot,\bs)$ and $B(\bmudot,\br)$.
%
%Therefore, a crystal isomorphism $\phi$ maps $B(\bla)$ to $B(\bmu)$, and the crystal graphs $B(\bla)$ and $B(\bmu)$ have the same arrows.
%Thanks to Relation (\ref{decompcrystal}), we can extend $\phi$ to the crystal of the whole Fock space,
%acting on its various connected components.
%We will thus sometimes write $\phi : B(\cF_\bs) \lra B(\cF_\br)$.
%By setting $\phi(\bs):=\br$, we can also denote $$\begin{array}{cccc} \phi : & \cF_\bs & \lra & \cF_{\phi(\bs)} \\ 
%& |\bla,\bs\rangle  & \longmapsto &  |\phi(\bla),\phi(\bs) \rangle .                         \end{array}$$

\medskip

\begin{defi}\label{defequivmp}
Let $|\bla,\bs\rangle \in \cF_\bs$ and $|\bmu, \br \rangle \in \cF_\br$.
We say that $|\bla,\bs\rangle$ and $|\bmu, \br \rangle$ (or simply $\bla$ and $\bmu$) are \textit{equivalent} if 
there is a $\Ueprime$-crystal isomorphism $\phi$ between $B(\bladot,\bs)$ and $B(\bmudot,\br)$ such that $\phi(|\bla,\bs\rangle) = |\bmu,\br\rangle$.
\end{defi}

\begin{rem}
In other terms, $\bla$ and $\bmu$ are equivalent if they appear at the same place in their respective crystal graphs.
\end{rem}

\medskip

The isomorphisms of $\Ueprime$-modules $V(\bs)\simeq V(\br)$ whenever $\br \in \bs \ \text{mod} \widehat{\fS_l}$ (cf. Section \ref{fockspace})
yield isomorphisms of crystal graphs between $B(\bs)$ and $B(\br)$.
% In particular, for any $\bs\in\Z^l$, there is a natural crystal isomorphism $B(\bs)\simeq B(\hbs)$ where $\hbs$ is the
% representative of $\bs$ lying in the fundamental domain $\sD_e$.
There exist also other natural crystal isomorphisms. 

\begin{propdef}\label{defcanonicaliso} Let $\bla \in \cF_\bs$.
There exists a unique $l$-charge $\br\in\sD_e$ and a unique FLOTW $l$-partition $\bmu\in\cF_\br$ such that
$|\bla,\bs\rangle$ and $|\bmu,\br\rangle$ are equivalent.
The associated $\Ueprime$-crystal isomorphism is called the \textit{canonical} crystal isomorphism.
\end{propdef}

\begin{proof}
First of all, if $\bla\in B(\bs)$ then, by the remark just above Proposition \ref{defcanonicaliso}, 
there is a crystal isomorphism between $B(\bs)$ and $B(\br)$ where $\br$ is the representative of $\bs$ in the fundamental domain $\sD_e$.

Suppose now that $|\bla,\bs\rangle\in\cF_\bs$ such that $B(\bla,\bs)\neq B(\bs)$. 
This means that $\bla$ (as a vertex in its crystal graph) is not in the connected component whose highest weight vertex is the empty multipartition.
Then there is a sequence $(i_1,\dots,i_p)$ such that $\te_{i_p}\dots \te_{i_1}(\bla) = \overset{\bullet}{\bla}$, the highest weight vertex in $B(\bla,\bs)$.
Write $\wt(\overset{\bullet}{\bla})=\sum_{i=0}^{e-1} a_i\La_i +d\de$ 
and define $\br$ to be the increasing $l$-charge containing $a_i$ occurences of $i$. In particular, $\br\in\sD_e$.
Then we have a natural crystal isomorphism $B(\bla,\bs) \overset{\phi}{\simeq} B(\br)$,
and therefore there is a FLOTW $l$-partition $\bmu:=\phi(\bla)=\tf_{i_1}\dots\tf_{i_p}(\bemptyset)$ in $B(\br)$ equivalent to $\bla$.

These elements are clearly unique, since $\sD_e$ is a fundamental domain for the action of $\widehat{\fS_l}$.
\end{proof}

The goal of this paper is to find a direct and purely combinatorial way to determine this canonical crystal isomorphism,
without having to determine the sequence of operators leading to the highest weight vertex and taking the reverse path in $B(\bemptyset,\br)$
as explained in the previous proof.
In the case where $\bla$ is a highest weight vertex, then the canonical crystal isomorphism is simply the so-called "peeling procedure" described in \cite{JaconLecouvey2012},
and maps $|\bla,\bs\rangle$ to $|\bemptyset,\br\rangle$, where $\br$ is the only charge in $\sD_e$ which is in the orbit of $\bs$ under the action of the extended
affine symmetric group.

\medskip

Let us also mention that in his paper \cite{Tingley2008}, Tingley describes a combinatorial procedure, called the tightening of a descending $l$-abacus.
Endowing the set of $l$-abacuses with a crystal structure, he then shows \cite[Lemma 3.12 and Theorem 3.14]{Tingley2008} that 
this tightening procedure is a crystal isomorphism, and moreover that this crystal, restricted to the set of tight descending $l$-abacuses with fixed $l$-charge (or compactification),
gives a realisation of the crystal of the associated abstract irreducible highest weight $\widehat{\fs\fl_e}$-module.
It is well-known that there is a one-to-one correspondence between (charged) $l$-partitions and $l$-abacuses (with origin), see \cite{JamesKerber1984}, or Uglov's paper
\cite{Uglov1999} for our purpose.
Unfortunately, Tingley's crystal structure on abacuses is somewhat different from Uglov's, which is the one that we use here and that is usually preferred
when it comes to interpretations in terms of modular representation theory of $G(l,1,n)$.
It should be interesting to replace Tingley's result in our context and see how this tightening procedure is possibly related to our canonical crystal isomorphism,
see also Remark \ref{remabacus}.

\section{The case $e=\infty$}

\label{e=inf}

In this section, we consider the particular (and easier) case where $e=\infty$. 
This means that we regard Fock spaces as $\Uinf$-modules, where $\Uinf$ is defined as the direct limit of the quantum algebras $\cU_q(\mathfrak{sl}_e)$.
We refer e.g. to \cite{Kac1984} for detailed background on $\Uinf$.
Actually, there is an action of $\Uinf$ on Fock spaces which generalises the action of $\Ueprime$ when $e$ tends to $\infty$.
In particular, $\cF_\bs$ is made into an integrable $\Uinf$-module, and all the properties of the $\Ueprime$-representation $\cF_\bs$ stated in Section \ref{fockspace}
still hold for the $\Uinf$-module structure.
With this point of view, the algebra $\Uinf$ is the natural way to extend $\Ueprime$ when $e\ra\infty$.

In this setting, the notion of being FLOTW for a multipartition simply translates to its symbol being semistandard,
with an increasing multicharge.
And as a matter of fact, we know a $\Uinf$-crystal isomorphism which associates to each multipartition a new multipartition whose symbol is semistandard.
This is the point of the following section.

\subsection{Schensted's bumping algorithm and solution of the problem}

\label{rs}

Let $\bla\in\cF_\bs$.

We first introduce the reading of the symbol $\fB_\bs(\bla)$.
It is the word obtained by writing the elements of $\fB_\bs(\bla)$ from right to left, starting from the top row.
Denote it by $\reading(\bla,\bs)$.
The Robinson-Schensted insertion procedure (or simply Schensted procedure, or bumping procedure) enables to construct a semistandard symbol starting from such a word.
We only recall it on an example (\ref{exars}) below, see also Example \ref{exainvertiblityrs} in Section \ref{hw}.
For proper background, the reader can refer to e.g. \cite{Fulton1997} or \cite{Lothaire2002}.
Denote by $\sP(\reading(\bla,\bs))$ the semistandard symbol obtained from $\reading(\bla,\bs)$ by applying this insertion procedure.
Finally, we set $\rs(\bs)$ and $\rs(\bla)$ to be the FLOTW multicharge and multipartition determined by $\fB_{\rs(\bs)}(\rs(\bla))= \sP(\reading(\bla,\bs))$

We further write $\rs$ for the map $$\begin{array}{cccc} \rs : & B(\bladot,\bs) & \lra & B(\overset{\bullet}{\rs(\bla)},\rs(\bs)) \\
                                                                     &  |\bla,\bs\rangle &  \longmapsto  &  | \rs(\bla), \rs(\bs) \rangle. \end{array}$$

\begin{thm}\label{thmrs}
$|\bla,\bs\rangle$ and $|\bmu,\br\rangle$ are equivalent if and only if 
$\sP(\reading(\bla,\bs))=\sP(\reading(\bmu,\br))$.
\end{thm}

For a proof of this statement, see for instance \cite[Section 3]{LascouxLeclercThibon1995} or \cite{Lecouvey2007}, 
which state the result for $\cU_q(\mathfrak{sl}_e)$-crystals, relying on the original arguments of Kashiwara in \cite{Kashiwara1991} and \cite{KashiwaraNakashima1994}.
Moreover, since the symbol associated to $|\rs(\bla),\rs(\bs)\rangle$ is semistandard, we have the following result.

\begin{cor}\label{corrs}
 $\rs$ is the canonical $\Uinf$-crystal isomorphism.
\end{cor}

\medskip

\begin{exa}\label{exars}
 $\bs=(0,2,-1)$ and $\bla=(2.1,3,4.1^2)$.

\medskip

Then $\fB_\bs(\bla)=\begin{pmatrix}
                     0 & 2 & 3 & 7 &   &   &  \\
		     0 & 1 & 2 & 3 & 4 & 5 & 9\\
		     0 & 1 & 2 & 4 & 6 &   &  
                    \end{pmatrix}$.

\medskip

The associated reading is $\reading(\bla)=7 3 2 0 9 5 4 3 2 1 0 6 4 2 1 0$, and the Schensted algorithm yields

\medskip

$\sP(\reading(\bla))=\begin{pmatrix}
       0 & 1 & 2 & 3 & 4 & 5 & 7 \\
      0 &  1 & 2 & 3 & 6 & 9 & \\
      0 & 2 & 4  &   &   &   &
      \end{pmatrix}$.

\medskip

Hence $=\rs(\bs)=(-2,1,2)$ and $=\rs(\bla)=(2.1,4.2,1)$.

\end{exa}

\medskip

\begin{prope}\label{propertyrs}
Suppose $\bs$ is such that $s_1\leq \dots \leq s_l$. Take $\bla\in\cF_\bs$.
Then $|\rs(\bla)| \leq |\bla|$. Moreover, $|\rs(\bla)|=|\bla|$ if and only if $\rs(\bla)=\bla$.
\end{prope}

\begin{proof}
Because of Corollary \ref{corrs}, we know in particular that $\rs(\bla)$ is in the connected component of $B(\cF_{\rs(\bs)})$
whose highest weight vertex is $\bemptyset$.
Hence, if we write $\bemptyset = \te_{i_m}\dots \te_{i_1} (\rs(\bla))$, we have $|\rs(\bla)|= m $.
Now, because $\rs$ is a crystal isomorphism, we have $\overset{\bullet}{\bla}=\te_{i_m}\dots \te_{i_1}(\bla)$.
If $\bla \neq \rs(\bla)$, then the symbol of $|\bla,\bs\rangle$ is not semistandard,
and the fact that $s_1\leq\dots \leq s_l$ ensures that $\overset{\bullet}{\bla}\neq \bemptyset$.
Therefore, we have $|\bla|= |\overset{\bullet}{\bla}| + m > m = |\rs(\bla)|$.
\end{proof}

\subsection{Another $\Uinf$-crystal isomorphism}

\label{rmatrix}

Let $\si\in\fS_l$, and for $\bs=(s_1,\dots, s_l)\in\Z^l$ denote $\bs^\si=(s_{\si(1)},\dots, s_{\si(l)})$.

% According to \cite[Section 2]{JaconLecouvey2010}, we know an explicit combinatorial description of the following $\Uinf$-crystal isomorphism:
$$ \begin{array}{cccc} 
    \chi_\si : &  B(\bs) & \lra & B({\bs^\si}) \\
           & \bla & \longmapsto & \chi_\si(\bla)
   \end{array}$$
   
In fact, in \cite[Corollary 2.3.3]{JaconLecouvey2010}, the map $\chi_\si$ is described in the case where $\si$ is a transposition $(c,c+1)$.
We do not recall here the combinatorial construction of $\chi_\si(\bla)$, since it is not really important for our purpose.
However, we notice the following property. It will be used in the proof that the algorithm we construct in Section \ref{reduction} terminates.

\begin{prope}\label{propertyrmatrix}
For all $\si\in\fS_l$, and for all $\bla \in V(\bs)$, $$|\chi_\si(\bla)| = |\bla|.$$
\end{prope}

We also denote simply by $\chi$ the isomorphism corresponding to a permutation $\si$ verifying $s_{\si(1)}\leq s_{\si(2)} \leq \dots \leq s_{\si(l)}$
(i.e. the reordering of $\bs$).

\section{General case and reduction to cylindric multipartitions}

\label{reduction}

\subsection{Compatibility between $\Ueprime$-crystals and $\Uinf$-crystals}

\label{compat}

In this section, we use the subscript or superscript $e$ or $\infty$ 
to specify which module structure we are interested in, in particular for the (reduced) $i$-word, crystal operators, crystal graph.

\medskip

The aim is to show that any $\Uinf$-crystal isomorphism is also a $\Ueprime$-crystal isomorphism.
This comes as a natural consequence of the existence of an embedding of $B_e(\bladot)$ in $B_\infty(\bladot)$,
as explained in \cite[Section 4]{JaconLecouvey2010}.
Note that the embedding in our case will just map any $\bla\in B(\bladot,\bs)$ onto itself, unlike in \cite{JaconLecouvey2010}.

\begin{lem}\label{lemcompat1} Let $i\in\llb 0, e-1 \rrb$.
Suppose there is an arrow $\bla \overset{i}{\lra} \bmu$ in the crystal graph $B_e(\bla)$,
and denote $\ga:=[\bmu]\backslash [\bla]$. 
Then there is an arrow $\bla \overset{j}{\lra} \bmu$ in the crystal graph $B_\infty(\bla)$, where $j=\cont(\ga)$.
\end{lem}

\begin{proof}
Denote by $w_i^e(\bla)$ (resp. $w_i^\infty(\bla)$) the $i$-word for $\bla$ with respect to the $\Ueprime$-crystal (resp. $\Uinf$-crystal) structure.
Then $w_i^e(\bla)$ is the concatenation of $i$-words for the $\Uinf$-crystal structure, precisely
 \begin{equation}\label{iwords} w_i^e(\bla) = \prod_{k\in\Z} w_{i+ke}^\infty (\bla). \end{equation}

We further denote $\hat{w}_i^e(\bla)$ and $\hat{w}_i^\infty(\bla)$ the reduced $i$-words (that is, after recursive cancellation of the factors $RA$).
The node $\ga$ is encoded in both $w_i^e(\bla)$ and $w_j^\infty(\bla)$ by a letter $A$.
Now if this letter $A$ does not appear in $\hat{w}_j^\infty(\bla)$, this means that there is a letter $R$ in $w_j^\infty(\bla)$ which simplifies with this $A$.
Hence, because of (\ref{iwords}), this letter $R$ also appears in $w_i^e(\bla)$ and simplifies with the $A$ encoding $\ga$, and 
$\ga$ cannot be the good addable $i$-node of $\bla$ for the $\Ueprime$-crystal structure, whence a contradiction.
Thus $\ga$ produces a letter $A$ in $\hat{w}_j^\infty(\bla)$.

In fact, this letter $A$ is the rightmost one in $\hat{w}_j^\infty(\bla)$. 
Indeed, suppose there is another letter $A$ in $\hat{w}_j^\infty(\bla)$ to the right of the $A$ encoding $\ga$.
Then it also appears in $\hat{w}_i^e(\bla)$ at the same place (again because of Relation (\ref{iwords})).
This contradicts the fact that $\ga$ is the good addable $i$-node for the $\Ueprime$-crystal structure.

Therefore, $\ga$ is the good addable $j$-node of $\bla$ for the $\Uinf$-crystal structure.

\end{proof}

\begin{lem}\label{lemcompat2}
Let $i\in\llb 0, e-1 \rrb$, and let $\varphi$ be a $\Uinf$-crystal isomorphism.
Suppose there is an arrow $\bla \overset{i}{\lra} \bmu$ in the crystal graph $B_e(\bla)$,
and denote $\ga:=[\bmu]\backslash [\bla]$. 
Then 
\begin{enumerate}
\item there is an arrow $\varphi(\bla) \overset{i}{\lra} \bnu$ in the crystal graph $B_e(\varphi(\bla))$,
\item $\bnu=\tf_j^\infty(\varphi(\bla))$, where $j=\cont(\ga)$.
\end{enumerate}
\end{lem}

\begin{proof}
First, for all $k$, we have the following relation:
\begin{equation}\label{iwords1}
\hat{w}_k^\infty(\bla) = \hat{w}_k^\infty(\varphi(\bla)).
\end{equation}
Indeed, if $\hat{w}_k^\infty(\bla) = A^\al R^\be$, then $\al$ can be seen as the number of consecutive arrows labeled by $k$ in $B_\infty(\bla)$ starting from $\bla$,
and $\be$ as the number of consecutive arrows labeled by $k$ leading to $\bla$.
Subsequently, the integers $\al$ and $\be$ are invariant by $\varphi$, and the relation (\ref{iwords1}) is verified.
Hence, by concatenating, we get 
\begin{equation}\label{iwords1bis}
\prod_{k\in\Z} \hat{w}_{i+ke}^\infty(\bla) = \prod_{k\in\Z} \hat{w}_{i+ke}^\infty(\varphi(\bla)),
\end{equation}
and therefore \begin{equation}\label{iwords2}\hat{w}_i^e(\bla) = \hat{w}_i^e(\varphi(\bla)),\end{equation}
which proves the first point.

Besides, we know by Lemma \ref{lemcompat1} that $\tf_i^e$ acts like $\tf_j^\infty$ on $\bla$.
Together with (\ref{iwords1bis}), this implies that $\tf_i^e$ acts like $\tf_j^\infty$ on $\varphi(\bla)$.
In other terms, $\bnu = \tf_j^\infty(\varphi(\bla))$, and the second point is proved.
\end{proof}

\begin{prop}\label{propcompat}
Every $\Uinf$-crystal isomorphism is also a $\Ueprime$-crystal isomorphism.
\end{prop}

\begin{proof}
The fact that $\varphi$ is a $\Uinf$-crystal isomorphism is encoded in the following diagram:

$$ \xymatrix@!0 @R=1.5cm @C=2cm{
  \bla  \ar[r]^-{\varphi}\ar[d]_-{\tf_j^\infty}
&
  \varphi(\bla) 
  \ar[d]^-{\tf_j^\infty} 
\\
    \bmu \ar[r]^-{\varphi}
&
   \varphi(\bmu)
  }$$

The first point of Lemma \ref{lemcompat2} tells us that we have:

$$ \xymatrix@!0 @R=1.5cm @C=2cm{
  \bla  \ar[r]^-{\varphi}\ar[d]_-{\tf_i^e}
&
  \varphi(\bla) 
  \ar[d]^-{\tf_i^e} 
\\
    \bmu
&
   \bnu 
  }$$

Besides,  $$\begin{array}{ccll} \bnu & =  & \tf_j^\infty(\varphi(\bla)) & \text{\quad by Point 2 of Lemma \ref{lemcompat2} } \\
                                & = &  \varphi(\tf_j^\infty(\bla)) & \text{\quad because $\varphi$ is a $\Uinf$-crystal isomorphism} \\
                                & =  & \varphi(\bmu). & 
           \end{array}$$

Hence we can complete the previous diagram in 

$$ \xymatrix@!0 @R=1.5cm @C=2cm{
  \bla  \ar[r]^-{\varphi}\ar[d]_-{\tf_i^e}
&
  \varphi(\bla) 
  \ar[d]^-{\tf_i^e} 
\\
    \bmu \ar[r]^-{\varphi}
&
   \varphi(\bmu)
  }$$
which illustrates the commutation between $\tf_i^e$ and $\varphi$.
\end{proof}

\medskip

As a consequence, the two particular $\Uinf$-crystal isomorphisms $\rs$ and $\chi_\si$, defined respectively in Section \ref{rs} and \ref{rmatrix},
are $\Ueprime$-crystal isomorphisms (for all values of $e\in\N_{>1}$).

\subsection{The cyclage isomorphism}

\label{cyclage}

One of the most natural $\Ueprime$-crystal isomorphisms to determine is the so-called \textit{cyclage} isomorphism.
For $\bs=(s_1,\dots,s_l)$, let $\bs':=(s_{l}-e,s_1,\dots,s_{l-1})$.
Then the following result is easy to show 
(see for instance \cite[Proposition 5.2.1]{JaconLecouvey2010}, or \cite[Proposition 3.1]{Jacon2007} for the simpler case $l=2$):

\begin{prop}\label{propcyclage}	
The map $$\begin{array}{cccc}
           \xi : & B(\bladot,\bs) & \lra & B(\cF_{\bs'}) \\
                 & (\la^1,\dots,\la^l)  & \longmapsto & (\la^l,\la^1,\dots,\la^{l-1})
          \end{array}$$
is a $\Ueprime$-crystal isomorphism. It is called the \textit{cyclage} isomorphism.
\end{prop}

Therefore, in the sequel, we denote $\xi(\bs):=(s_{l}-e,s_1,\dots,s_{l-1})$ and $\xi(\bla):=(\la^l,\la^1,\dots,\la^{l-1})$.

\medskip

\begin{rem}\label{remcyclage0}
Actually, we have more than this.
Indeed, the map $\xi$ is clearly invertible.
Hence, because $\te_i\circ\tf_i = \tf_i\circ\te_i = \Id$ whenever they act non trivially,
we have $$\begin{array}{rccc}
         & \tf_i \circ \xi & = & \xi \circ \tf_i \\
         \text{i.e.} \quad  & \xi & = & \te_i \circ \xi \circ \tf_i \\
         \text{i.e.} \quad & \xi \circ \te_i & = & \te_i \circ \xi.
        \end{array}$$
        
Hence, $\te_i$ and $\xi$ also commute.

\end{rem}

\medskip

\begin{rem}
Note that in \cite{JaconLecouvey2010}, the cyclage is defined slightly differently, namely by 
$$\begin{array}{cccc}
   \ze : & B(\bladot,\bs) & \lra & B(\cF_{(s_2,s_3,\dots, s_l, s_{1+e})}) \\
            &      (\la^1,\dots,\la^l)  & \longmapsto & (\la^2,\dots,\la^{l},\la^1).
  \end{array}$$
It is easy to see that both definitions are equivalent. Indeed, one recovers $\xi$ by:
\begin{enumerate}
\item applying $l-1$ times $\ze$,
\item translating all components of the multicharge $\ze^{l-1}(\bs)$ by $-e$ 
(which is a transformation that has clearly no effect on the multipartition).
\end{enumerate}

\end{rem}

\medskip

\begin{rem}\label{remcyclage1}
Note that the cylindricity condition defined in Definition \ref{defflotw} is conveniently expressible in terms of symbols using the cyclage $\xi$.
Precisely, $|\bla,\bs\rangle$ is cylindric if and only if the three following conditions are verified:
\begin{enumerate}
 \item $\bs\in\sS_e$,
 \item $\fB_\bs(\bla)$ is semistandard,
\item $\fB_{\xi(\bs)}(\xi(\bla))$ is semistandard.
\end{enumerate}
 
\end{rem}

\begin{rem}\label{remcyclage2}
We will see (Remark \ref{remisopsi2}) that the canonical isomorphism we aim to determine can be naturally regarded as a generalisation of this cyclage isomorphism...
\end{rem}

Finally, it is straightforward from the definition of $\xi$ that the following property holds:

\begin{prope}\label{propertycyclage}
For all $\bla\in\cF_\bs$, we have $|\xi(\bla)|=|\bla|$.
\end{prope}

\subsection{Finding a cylindric equivalent multipartition}

\label{algorithm}

In this section, we make use of the $\Ueprime$-crystal isomorphisms $\rs$ (see Section \ref{rs}) and $\xi$ (defined in Proposition \ref{propcyclage})
to construct an algorithm which associates to any charged multipartition $|\bla,\bs\rangle$ an equivalent charged multipartition $|\bmu,\br\rangle$ 
which is cylindric (see Definition \ref{defflotw}).
In the sequel, we will denote by $\sC_\bs$ the subset of $\cF_\bs$ of cylindric $l$-partitions.
In particular, this implies that $\bs\in\sS_e$.
First of all, let us explain why restricting ourselves to cylindric multipartitions is relevant.

\begin{prop}\label{propstability} Let $\bs\in\sS_e$.
The set $\sC_\bs$ is stable under the action of the crystal operators.
\end{prop}

\begin{proof} 

Let $\bla\in\sC_\bs$.
By Remark \ref{remcyclage1}, we know that $\fB_\bs(\bla)$ is semistandard and $\fB_{\xi(\bs)}(\xi(\bla))$ is semistandard.

It is easy to see that $\fB_\bs(\tf_i(\bla))$ (resp. $\fB_\bs(\te_i(\bla))$) is still semistandard, whenever $\tf_i$ (resp. $\te_i$) acts non trivially on $\bla$.
Indeed, denote $\ga$ the good addable $i$-node and let $j=\cont(\ga)$.
It is encoded by an entry $j+p$, where $p$ is the size of the symbol, see Section \ref{generalities}.
By definition of being a good node is the leftmost $i$-node amongst all $i$-node of content $j$.
Hence, there is no other no entry below the entry $j+p$ encoding $\ga$.
Since $\te_i$ just turns this $j+p$ into $j+p+1$, the symbol of $\tf_i(\bla)$ is still semistandard.
The similar argument applies to $\te_i(\bla)$.

Since the symbol of $\xi(\bla)$ is semistandard, by the same argument as above, we deduce that the symbol of $\tf_i(\xi(\bla))$ is still semistandard,
i.e. the symbol of $\xi(\tf_i(\bla))$ is semistandard, since $\xi$ commutes with $\tf_i$ (by Proposition \ref{propcyclage}).
This result also holds for $\xi(\te_i(\bla))$ because $\xi$ also commutes with $\te_i$ (see Remark \ref{remcyclage0}).

By Remark \ref{remcyclage1}, this means that $\tf_i(\bla)$ and $\te_i(\bla)$ are both cylindric.

\end{proof}

\medskip

The algorithm expected can now be stated.
Firstly, if $\fB_\bs(\bla)$ is not semistandard, then we can apply $\rs$ to get a charged multipartition whose symbol is semistandard.
Hence we can assume that $\fB_\bs(\bla)$ is semistandard. In particular, this implies that $s_c \leq s_{c+1}$ for all $c\in\llb 1, l-1 \rrb$.

Then,

\begin{enumerate}
 \item If $s_l-s_1<e$, then:
 \begin{enumerate}
 \item if $\fB_{\xi(\bs)}(\xi(\bla))$ is semistandard, then $|\bla,\bs\rangle$ is cylindric, hence we stop and take $\bmu=\bla$ and $\br=\bs$.
 \item if $\fB_{\xi(\bs)}(\xi(\bla))$ is not semistandard, then put $\bla \leftarrow \rs(\xi(\bla))$ and $\bs \leftarrow \rs(\xi(\bs))$ and start again.
 \end{enumerate}
 \item If $s_l-s_1 \geq e$, then put  $\bla \leftarrow\rs(\xi(\bla))$ and $\bs \leftarrow \rs(\xi(\bs))$ and start again.
\end{enumerate}

\begin{prop}\label{propalgorithm}
 The algorithm above terminates.
\end{prop}

\begin{proof}
For a multicharge $\bs=(s_1,\dots s_l)$, we denote $||\bs||:= \sum_{k=2}^l (s_k-s_1)$.
Hence, if $\fB_\bs(\bla)$ is semistandard, we have $||\bs||\geq 0$. 
In particular, at each step in the algorithm, this statistic is always non-negative, since we replace $\bs$ by $\rs(\xi(\bs))$.

Suppose we are in case 1.(b). Since $s_l-s_1<e$ and $s_1\leq s_2\leq\dots\leq s_l$, we have 
$(s_l-e)<s_1\leq s_2 \leq \dots \leq s_{l-1}$. In other terms, the multicharge $\xi(\bs)=(s_l-e,s_1 , \dots, s_{l-1})$ is an increasing sequence.
Hence Property \ref{propertyrs} applies, and we have $|\rs(\xi(\bla))|<|\bla|$.

Suppose we are in case 2. 
The first thing to understand is that  we get the same multipartition and multicharge applying $\xi$ and $\rs$, 
or applying $\xi$, then $\chi$ (see Section \ref{rmatrix}) and $\rs$.
Indeed, $\chi$ just reorders the multicharge and gives the associated multipartition, which is a transformation already included in $\rs$,
(which gives a multipartition whose symbol is semistandard). 
Hence, we consider that $\rs(\xi(\bla))$ is obtained by applying successively $\xi$, then $\chi$, and finally $\rs$, to $\bla$.
In this procedure, it is possible that $\rs$ acts trivially (i.e. that $\chi(\xi(\bla))$ is already semistandard).
In fact,
\begin{itemize}
 \item If $\rs$ acts non trivially, then on the one hand $\chi(\xi(\bla))$ is non semistandard; 
and on the other hand $\chi(\xi(\bs))$ is an increasing sequence (by definition of $\chi$).
Thus, we have 
 $$\begin{array}{ccll} |\rs(\chi(\xi(\bla)))|  & <  & |\chi(\xi(\bla))|  & \text{ \quad applying Property \ref{propertyrs} }  \\
                                           & = &  |\bla| \hspace{0.82cm} & \text{ \quad by Properties \ref{propertyrmatrix} and \ref{propertycyclage} }.
   \end{array}$$
Hence in this case,  $|\rs(\xi(\bla))| < |\bla|$.
 \item If $\rs$ acts trivially, then this argument no longer applies.
 However, we have $|| \rs(\chi(\xi(\bs))) || = || \chi(\xi(\bs))) || < ||\bs ||$.
 Indeed, denote $\bs'= \chi(\xi(\bs))$. 
Since $s_l-s_1 \geq e $, we have $s_l-e \geq s_1$. 
This implies that the smallest element of $\xi(\bs)=(s_l-e,s_1,\dots,s_{l-1})$ is again $s_1$, and that $s'_1=s_1$.
Hence $$\begin{array}{cc>{\ds}l}  ||\bs'|| & =  & \sum_{k=2}^l (s'_k-s'_1) \\
                                & =   & \sum_{k=2}^l (s'_k- s_1) \\
                                &   = & \sum_{k=2}^{l-1} (s_k - s_1) + (s_l-e) - s_1 \\
                                & = & \sum_{k=2}^l (s_k-s_1) - e \\
                                &  = &  ||\bs || - e \\
                                & <  & ||\bs || \end{array}$$
Note also that in this case,  $|\rs(\xi(\bla))|= |\chi(\xi(\bla))|=|\bla|$ by Properties \ref{propertyrmatrix} and \ref{propertycyclage}.
\end{itemize}

We see that at each step, the rank $|.|$ can never increase.
In fact, since it is always non-negative, there is necessarily a finite number of steps at which this statistic decreases.
Moreover, when the rank does not increase, then the second statistic $||.||$ decreases.
Since it can never be negative (as noted in the beginning of the proof), there is also finite number of such steps.
In conclusion, there is a finite number of steps in the algorithm, which means that it terminates.

\end{proof}

\begin{rem}\label{remalgorithm}
This algorithm can also be stated in the simpler following way:
\begin{enumerate}
\item If $|\bla,\bs\rangle$ is cylindric, then stop and take $\bmu=\bla$ and $\br=\bs$.
\item Else, put $\bla \leftarrow \rs(\xi(\bla))$ and $\bs \leftarrow \rs(\xi(\bs))$ and start again.
\end{enumerate}

\medskip

\end{rem}

In other terms, we have proved that for each charged $l$-partition $|\bla,\bs\rangle$, there exists $m\in\N$ such that 
$((\rs\circ \xi)^m\circ \rs) \; (|\bla,\bs\rangle)$ is cylindric.
This integer $m$ a priori depends on $\bla$. The following proposition claims that it actually does not depend on $\bla$,
but only on the connected component $B(\bladot,\bs)$.

\begin{prop}\label{mconstant}
Let $B(\bladot,\bs)$ be a connected component of $B(\cF_\bs)$.
Then there exists $m\in\N$ such that for all $\bla\in B(\bladot,\bs)$,
\begin{itemize}
 \item $((\rs\circ \xi)^{m} \circ \rs) \; (|\bla,\bs\rangle)$ is cylindric, and
\item $((\rs\circ \xi)^{m'} \circ \rs) \; (|\bla,\bs\rangle)$ is not cylindric for all $m'<m$.
\end{itemize}
\end{prop}

\begin{proof}
Because of Proposition \ref{propalgorithm}, we know that
for each $\bla\in B(\bladot,\bs)$, there exists $m(\bla)\in\N$ verifying this property. Write $m_0 = m(\bladot)$.
Now, take any $\bla\in B(\bladot,\bs)$, and write $\bla = \tf_{i_p} \dots \tf_{i_1} (\bladot)$.
Because $\rs$ and $\xi$ are crystal isomorphisms, they commute with the crystal operators $\tf_i$.
Hence
$$\begin{array}{rcl}
 ((\rs\circ\xi)^{m_0}\circ \rs) \; (|\bla,\bs\rangle) & = & ((\rs\circ\xi)^{m_0}\circ \rs) \; (\tf_{i_p} \dots \tf_{i_1} (|\bladot,\bs\rangle) ) \\
                       & = & \tf_{i_p} \dots \tf_{i_1} ( ((\rs\circ\xi)^{m_0}\circ \rs) \;(|\bladot,\bs\rangle) ) \\
                       & = & \tf_{i_p} \dots \tf_{i_1} (|\overset{\bullet}{\bmu},\br\rangle) \text{ \quad where $|\overset{\bullet}{\bmu},\br\rangle$ is cylindric} \\
                       & =: & |\bmu,\br \rangle \text{ \quad, which is cylindric because of Proposition \ref{propstability} }.
\end{array}$$
Moreover, is there exists $m'<m_0$ such that $((\rs\circ\xi)^{m'}\circ \rs) \; (|\bla,\bs\rangle)$ is cylindric, then by the same argument
$((\rs\circ\xi)^{m'}\circ \rs) \;(|\bladot,\bs\rangle)$ is cylindric, which contradicts the minimality of $m_0$.

Therefore $m(\bla)=m_0$.
\end{proof}

\newcommand{\moinstrois}{$-$3}
\newcommand{\moinsquatre}{$-$4}
\newcommand{\moinscinq}{$-$5}
\newcommand{\moinssix}{$-$6}
\newcommand{\moinssept}{$-$7}
\newcommand{\moinshuit}{$-$8}
\newcommand{\moinsneuf}{$-$9}
\newcommand{\moinsdix}{$-$10}
\newcommand{\moinsonze}{$-$11}
\newcommand{\moinsdouze}{$-$12}
\newcommand{\moinstreize}{$-$13}
\newcommand{\moinsquatorze}{$-$14}
\newcommand{\moinsquinze}{$-$15}

\setcounter{MaxMatrixCols}{20}

\begin{exa}\label{exaalgo}
Set $e=4$, $l=3$, $s=(0,9,5)$, and $\bla=(4.2^2.1^3,5.2^3.1^4,7^2.6.4^2.2^2.1^3) \in \cF_\bs$.
Firstly, we see that 
$$\fB_\bs(\bla)=
\begin{pmatrix}
 0 & 1 & 3 & 4 & 5 & 7 & 8 & 11 & 12 & 15 & 17 & 18 & & &  & \\
 0 & 1 & 2 & 3 & 4 & 5 & 6 & 7 & 9 & 10 & 11 & 12 &  14 & 15 & 16 & 20\\
  0 & 2 & 3 & 4 & 6 & 7 & 10 & & & & & & & &  &
\end{pmatrix}
$$ 
is not semistandard. Thus we first compute $\tbla:=\rs(\bla)$ and $\tbs:=\rs(\bs)$.
We obtain 
$$ \fB_{\tbs}(\tbla) =
\begin{pmatrix}
   0 & 1 & 2 & 3 & 4 & 5 & 6 & 7 & 8 & 10 & 11 & 12 & 14 & 15 & 16 & 17 & 18 \\
  0 & 1 & 3 & 4 & 5 & 7 & 9 & 11 & 12 & 15 & 20 & & & & & & \\
 0 & 2 & 3 & 4 & 6 & 7 & 10 & & & & & & & & & &
  \end{pmatrix},$$
i.e. $\tbla=(4.2^2.1^3,10.6.4^2.3.2.1^3,2^5.1^3)$ and $\tbs=(0,4,10)$.
We see that $|\tbla,\tbs\rangle$ is not cylindric.
Hence we compute $\bla^{(1)} := (\rs\circ \xi) (\bla)$ and $\bs^{(1)}:= (\rs\circ \xi) (\bla)$.
We get 
$$\fB_{\bs^{(1)}}(\bla^{(1)}) =
\begin{pmatrix}
 0 & 1 & 2 & 3 & 4 & 5 & 7 & 8 & 9 & 11 & 12 & 13 & 14 & 15 & 20 \\
 0 & 1 & 3 & 4 & 6 & 7 & 10 & 11 & 12 &&&&&& \\
 0 & 2 & 3 & 4 & 6 & 7 & 10 &&&&&&&&
\end{pmatrix},$$
i.e. $\bla^{(1)}=(4.2^2.1^3,4^3.2^2.1^2, 6.2^5.1^3)$ and $s^{(1)}=(0,2,8)$.

We keep on applying $\rs\circ\xi$ until ending up with a cylindric multipartition.
In fact, if we denote $\bla^{(k)}:= (\rs\circ\xi)^k(\bla)$ and $\bs^{(k)}:= (\rs\circ\xi)^k(\bs)$,
we can compute $|\bla^{(2)}, \bs^{(2)}\rangle$, $|\bla^{(3)}, \bs^{(3)}\rangle$, $|\bla^{(4)}, \bs^{(4)}\rangle$, and we finally have
$$\fB_{\bs^{(5)}}(\bla^{(5)})= 
\begin{pmatrix}
 0 & 1 & 2 & 3 & 4 & 6 & 7 & 9 & 11 & 12 \\
0 & 1 & 3 & 4 & 5 & 7 & 8 & 10 & 11 & 16 \\
0 & 2 & 3 & 4 & 6 & 7 & 10 &&& \\
\end{pmatrix},$$
i.e. $\bla^{(5)}=(4.2^2.1^3,7.3^2.2^2.1^3,3^2.2.1^2)$ and $\bs^{(5)}=(-4,-1,-1)$.
We see that $|\bla^{(5)}, \bs^{(5)}\rangle$ is cylindric.

This charged multipartition has the following Young diagram with contents:
$$|\bla^{(5)}, \bs^{(5)}\rangle = 
\left( \;\young(\moinsquatre\moinstrois\moinsdeux\moinsun,\moinscinq\moinsquatre,\moinssix\moinscinq,\moinssept,\moinshuit,\moinsneuf) \; 
, \; \young(\moinsun012345,\moinsdeux\moinsun0,\moinstrois\moinsdeux\moinsun,\moinsquatre\moinstrois,\moinscinq\moinsquatre,\moinssix,\moinssept,\moinshuit) \; 
, \; \young(\moinsun01,\moinsdeux\moinsun0,\moinstrois\moinsdeux,\moinsquatre,\moinscinq) \; \right).$$

With this representation, we see that $|\bla^{(5)}, \bs^{(5)}\rangle$ is not FLOTW (cf. Definition \ref{defflotw}).
Therefore, It remains to understand how to obtain an equivalent FLOTW multipartition from a cylindric multipartition.
This is the point of the next section, which contains the main result of this paper (Theorem \ref{thmisopsi}).

\end{exa}

\section{The case of cylindric multipartitions}

\label{cylindric}

Recall that for $\bs\in\sS_e$, we have denoted $\sC_\bs$ the set of cylindric $l$-partitions.

\subsection{Pseudoperiods in a cylindric multipartition}

\label{pp}

Let $|\bla,\bs\rangle \in \sC_\bs$ such that $\bla$ is not FLOTW.
Then there is a set of parts of the same size, say $\al$, such that the residues at the end of these parts cover $\llb 0 , e-1 \rrb$.
This is formalised in the following definition.

\begin{defi}\label{defpp} \hfil

\begin{itemize}
 \item The \textit{first pseudoperiod} of $\bla$ is the sequence $P(\bla)$ of its rightmost nodes $$\ga_1=(a_1,\al,c_1), \dots, \ga_e=(a_e,\al,c_e)$$ verifying:
\begin{enumerate}
\item there is a set of parts of the same size $\al\geq 1$ such that the residues at the rightmost nodes of these parts cover $\llb 0 , e-1 \rrb$.
\item $\al$ is the maximal integer verifying 1.
\item $\ds\cont(\ga_1)=\max_{\substack{c\in \llb 1,l \rrb \\ a \in \llb 1,\h(\la^c) \rrb}} \cont(a,\al,c)$ \quad  and \quad 
 $\ds c_1= \min_{\substack{a \in \llb 1,\h(\la^c) \rrb \\ \cont(a,\al,c)=\cont(\ga_1)}} c$,
\item 
for all $i\in\llb 2,e \rrb$, \quad 
$\ds \cont(\ga_i) =  
\max_{\substack{c\in \llb 1,l \rrb \\ a \in \llb 1,\h(\la^c) \rrb \\ \cont(a,\al,c)<\cont(\ga_{i-1})}} \cont(a,\al,c) $ 
\quad 
\newline 
and  \quad $\ds c_i  =   \min_{\substack{a \in \llb 1,\h(\la^c) \rrb \\ \cont(a,\al,c)=\cont(\ga_i)}} c$.
\end{enumerate}
In this case, $P(\bla)$ is also called a \textit{$\al$-pseudoperiod} of $\bla$, and $\al$ is called the \textit{width} of $P(\bla)$.

\item Denote $\bla^{[1]}:=\bla\backslash P(\bla)$,
that is the multipartition obtained by forgetting 
\footnote{This means that one considers only the nodes of $\bla$ that are in parts whose rightmost node is not in $P(\bla)$,
but without changing the indexation nor the contents of these nodes.}
in $\bla$ the parts $\al$ whose rightmost node belongs to $P(\bla)$.
Let $k\geq 2$. Then the \textit{$k$-th pseudoperiod of $\bla$} is defined recursively as being 
the first pseudoperiod of $\bla^{[k]}$, if it exists, where $\bla^{[k]}:=\bla^{[k-1]}\backslash P(\bla^{[k-1]})$.
\end{itemize}

Any $k$-th pseudoperiod of $\bla$ is called a \textit{pseudoperiod} of $\bla$.

\end{defi}

\begin{exa}\label{exapp}
Let $e=3$, $\bs=(2,3,4)$, and $\bla=(2.1^2,2.1^3,2.1^4)$.
One checks that $\bla$ is cylindric for $e$ but not FLOTW.
Then $\bla$ has the following diagram with contents:

\newcommand{\bluethree}{\textcolor[rgb]{0,0,1}{3}}
\newcommand{\bluefour}{\textcolor[rgb]{0,0,1}{4}}
\newcommand{\bluefive}{\textcolor[rgb]{0,0,1}{5}}
\newcommand{\redone}{\textcolor[rgb]{1,0,0}{1}}
\newcommand{\redtwo}{\textcolor[rgb]{1,0,0}{2}}
\newcommand{\redthree}{\textcolor[rgb]{1,0,0}{3}}
\newcommand{\greenzero}{\textcolor[rgb]{0,1,0}{0}}
\newcommand{\greenone}{\textcolor[rgb]{0,1,0}{1}}
\newcommand{\greentwo}{\textcolor[rgb]{0,1,0}{2}}

$$ \bla= \left( \; \young(2\bluethree,\redone,\greenzero) \; , \; \young(3\bluefour,\redtwo,\greenone,0) \; , \; \young(4\bluefive,\redthree,\greentwo,1,0) \;\right).$$

Then $\bla$ has one $2$-pseudoperiod and two $1$-pseudoperiods. Its first pseudoperiod consists of
$\ga_1=(1,2,3), \ga_2=(1,2,2)$ and $\ga_3=(1,2,1)$, with respective contents $5$, $4$ and $3$, colored in blue.
The second pseudoperiod is $(\ga_1=(2,1,3), \ga_2=(2,1,2), \ga_3=(2,1,1) )$, with red contents; and the third (and last) pseudoperiod is
$(\ga_1=(3,1,3), \ga_2=(3,1,2), \ga_3=(3,1,1) )$, with green contents.
\end{exa}

\begin{lem}\label{lempp1}
\begin{enumerate}
 \item $\cont(\ga_i)=\cont(\ga_{i-1})-1$ for all $i \in\llb 2, e \rrb$.
 In other terms, the contents of the elements of the pseudoperiod are consecutive.
 \item $c_i \leq c_{i-1}$ for all $i \in\llb 2, e \rrb$.
\end{enumerate}

\end{lem}

\begin{proof}
\begin{enumerate}
 \item Suppose there is a gap in the sequence of these contents. Then the pseudoperiod must spread over $e+1$ columns in the symbol $\fB_\bs(\bla)$.
Denote by $\fb$ the integer of $\fB_\bs(\bla)$ corresponding to the last element of $P(\bla)$, and $k$ the column where it appears.
The integer of $\fB_\bs(\bla)$ corresponding to the gap must be in column $k+1$, and since $\bla$ is cylindric, it must be greater than or equal to $\fb+e$.
In fact, it cannot be greater than $\fb+e$ since the corresponding part is below a part of size $\al$, and it has to correspond to a part of size $\al$, 
and there cannot be a gap, whence a contradiction.

\item Since the nodes of $P(\bla)$ are the rightmost nodes of parts of the same size $\al$, 
together with the fact $s_1\leq \dots \leq s_l$, and point 1., $\ga_i$
is necessarily either to the left of $\ga_{i-1}$ or in the same component.

\end{enumerate}
\end{proof}

\begin{rem}\label{rempp2}
If $\al=\max_{i,j} \la^i_j$, then the first pseudoperiod corresponds to a "period" in $\fB_\bs(\bla)$, accordingly to \cite[Definition 2.2]{JaconLecouvey2012}
This is the case in Example \ref{exapp}.
In the case where each pseudoperiod corresponds to a period in the symbol associated to $\bla^{[k]}$,
one can directly recover the empty $l$-partition and the corresponding multicharge using the "peeling procedure" explained in \cite{JaconLecouvey2012}.
However, in general, $\fB_\bs(\bla)$ might not have a period, as shown in the following example.
\end{rem}

\begin{exa}\label{exapp2}
 $e=4, \;\; \bs=(5,6,8) \mand \bla=(6^2.2.1,3.2^3.1^2,6.2^2.1^3)$. Then $\bla\in\sC_\bs$ but is not FLOTW for $e$.
It has the following Young diagram with residues:

$$\bla=\left( \; \young(56789\dix,456789,34,2,1) \; , \; \young(678,56,45,34,2,1) \; , \; \young(89\dix\onze\douze\treize,78,67,5,4,3) \; \right)$$

Then there is a $2$-pseudoperiod and a $1$-pseudoperiod.
The first pseudoperiod of $\bla$ consists of the nodes $\ga_1=(2,2,3)$, $\ga_2=(3,2,3)$, $\ga_3=(2,2,2)$ and $\ga_4=(3,2,2)$,
with respective contents $8$, $7$, $6$ and $5$.
The $1$-pseudoperiod is $((4,1,3),(5,1,3),(6,1,3),(4,1,1))$.

\end{exa}

Of course, one could also describe pseudoperiods on the $\bs$-symbol associated to $\bla$.
However, this approach is not that convenient, and in the setting of cylindric multipartitions,
we favour the "Young diagram with contents" approach, which encodes the same information.
Nevertheless, we notice this property, which will be used in the proof of Lemma \ref{lemisopsi1}:

\medskip

\begin{prop} \label{proppp}
Let $P(\bla)$ be a pseudoperiod of $\bla$. 
Denote by $B$ the set of entries of $\fB_\bs(\bla)$ corresponding to the nodes of $P(\bla)$.
Then each column of $\fB_\bs(\bla)$ contains at most one element of $B$.
Moreover, the elements of $B$ appear in consecutive columns of $\fB_\bs(\bla)$.
\end{prop}

\begin{proof}
This is direct from the fact that the nodes of $P(\bla)$ are all rightmost nodes of parts of the same size $\al$,
together with Lemma \ref{lempp1}.
\end{proof}

\bigskip

We will now determine the canonical $\Ueprime$-crystal isomorphism for cylindric multipartitions.
In the following section, we only determine the suitable multicharge. In Section \ref{isopsi}, we explain how to construct the actual corresponding FLOTW multipartition.

\subsection{Determining the multicharge}

\label{mc}

Take $\bs\in\sS_e$ and $|\bla,\bs\rangle \in \sC_\bs$.
By the discussion in the proof of Proposition \ref{defcanonicaliso}, there is only one $l$-charge $\varphi(\bs)\in\sD_e$ such that $|\bla,\bs\rangle$ and $|\bemptyset,\varphi(\bs)\rangle$
are equivalent, namely the representative of $\bs$ in the fundamental domain $\sD_e$ modulo $\widehat{\fS}_l$.

Besides, since $\bs\in\sS_e$, it is clear that $\varphi(\bs)$ is obtained from $\bs$ only by applying a certain number of times the operator $\xi$ or $\xi^{-1}$.
In other terms, we have:

\begin{prop}\label{propmc}
Let $\bs\in\sS_e$.
There exists $k\in\Z$ such that $\xi^k(\bs) \in \sD_e$. 
In fact, we have $\varphi(\bs) = \xi^k(\bs)$.
\end{prop}

Just to be consistent, let us recall that in \cite{JaconLecouvey2012}, Jacon and Lecouvey have proved that when $\bla=\overset{\bullet}{\bla}$ is a highest vertex, then
it suffices to "peel" the symbol of $\bla$ in order to get an empty equivalent multipartition.
We do not recall here this procedure in detail, but it basically consists in removing all periods in the symbol of $\bla$ (see Example \ref{exapeel}).
% When we start from a multipartition $\bla$ which is no longer a highest weight vertex, we can, in general, no longer peel the symbol,
% for it does not necessarily contain a period anymore (see Remark \ref{rempp2}).
% However, the multicharge we look for is constant along the crystal, hence entirely determined by the highest weight vertex $\overset{\bullet}{\bla}$.
Therefore, $\varphi(\bs)$ is also the representative in $\sD_e$ of the multicharge associated to the empty multipartition obtained after peeling $\overset{\bullet}{\bla}$.
In the rest of this paragraph, we show that both construction of $\varphi(\bs)$ coincide.

Firstly, since $|\overset{\bullet}{\bla},\bs\rangle$ is already cylindric, it turns out that the period can be easily read in the symbol of $\overset{\bullet}{\bla}$.
In fact, we have the following property:

\begin{prop}\label{cylmc1}
 Let $|\bla,\bs\rangle \in \sC_\bs$ such that $\fB_\bs(\bla)$ has a period $P$.
 Then each of the rightmost $e$ columns of $\fB_\bs(\bla)$ contains a unique element of $P$.
\end{prop}

\begin{proof}
Since a period is nothing but a pseudoperiod whose width is the largest part in $\bla$ (cf. Remark \ref{rempp2}),
this result is just a particular case of Proposition \ref{proppp}.
\end{proof}

Hence, deleting the first period (first step of the peeling) $\overset{\bullet}{\bla}$, we get a multicharge $\bs^{(1)}$ verifying
$$\begin{array}{cl}
s^{(1)}_l & =s_l-(s_l-s_{l-1}) = s_{l-1} \\
s^{(1)}_{l-1} & = s_{l-1}- (s_{l-1}-s_{l-2}) = s_{l-2} \\
\vdots \\
s^{(1)}_2 & = s_2 - ( s_2-s_1) =s_1 \\
s^{(1)}_1 & = s_1 - ( e- (s_l - s_1))=s_l-e. \\
\end{array}$$
In other terms, we have $\bs^{(1)}=\xi(\bs)$, where $\xi$ is the cyclage operator defined in Section \ref{cyclage}.

\medskip

Applying this recursively, we obtain

\begin{prop}\label{propmc0}
For all $r\geq 1$, $\bs^{(r)}=\xi^r(\bs),$ 
where $\bs^{(r)}$ denotes the multicharge associated with the peeled symbol after $r$ steps (with $\bs^{(0)}=\bs$).
\end{prop}

\begin{rem}\label{remmcpeel}
It is possible that, at some step, a multicharge $\bs^{(r)}$ will not be in $\sS_e$ anymore.
However, for any $r$, one always has $s^{(r)}_i \leq s^{(r)}_j$ for $i<j$, and $s^{(r)}_l-s^{(r)}_1 \leq e$.
Moreover, if $s^{(r)}_l-s^{(r)}_1=e$, then $s^{(r+p)}_l-s^{(r+p)}_1 <e$, 
where $p\geq 1$ is the number of components of $s^{(r)}$ equal to $s^{(r)}_l$.

Note that this never happens if $s_i<s_j$ for all $i<j$.
\end{rem}

\begin{exa}\label{exapeel}
 $e=3$, $\bs=(3,3,4)$, $\bla=(3^2.2,2^2.1,3.1^2)=\overset{\bullet}{\bla}$. One checks that $|\bla,\bs\rangle$ is cylindric but not FLOTW.
 The associated symbol is $$\begin{pmatrix}
                             0 & 1 & 3 & 4 & 7 \\
			     0 & 2 & 4 & 5 &    \\
                             0 & 3 & 5 & 6 &  
                            \end{pmatrix} $$

Peeling this symbol, we get successively

$$\begin{pmatrix}
 0 & 1 & 3 & 4 \\
0 & 2 & 4 & 5 \\
0 & 3
 \end{pmatrix} \mand \bs^{(1)}=(1,3,3),$$

$$\begin{pmatrix}
0 & 1 & 3 & 4 \\
0 & 2 &  &  \\
0 & 
 \end{pmatrix}
\mand \bs^{(2)}=(0,1,3),$$

$$\begin{pmatrix}
0 & 1 \\
0 &  \\
0 & 
 \end{pmatrix}
\mand \bs^{(3)}=(0,0,1).$$

Note that $\bs^{(2)} \notin \sS_e$.
\end{exa}

\medskip

This yields an alternative proof of Proposition \ref{propmc}:
\begin{proof}
We know that $\varphi(\bs)$ is the representative of the multicharge $\bs'$ associated to the peeled symbol of $\overset{\bullet}{\bla}$.
Because of Consequence \ref{propmc0}, $\bs'$  is obtained from $\bs$ by applying several times 
(say $t$ times) $\xi$. In other terms, $\bs'=\xi^t(\bs)$.
Now, 
\begin{itemize}
 \item if $\bs'\in\sD_e$, then $\varphi(\bs) = \bs'$ and $k=t$.
 \item if $\bs'\in\sS_e\backslash\sD_e$, then the representative of $\bs'$ in $\sD_e$ is of the form $\xi^v(\bs')$ for some $v\in\Z$.
 Hence, $\varphi(\bs)=\xi^{k}(\bs)$ with $k=t+v$.
 \item if $\bs'\notin\sS_e$, we are however ensured (see Remark \ref{remmcpeel}) that there exists $p\in\llb 1 ,e-1 \rrb$ such that $\xi^p(\bs')\in\sS_e$.
We are then in the previous situation, i.e. there exists $v\in\Z$ such that $\xi^{p+v}(\bs')\in\sD_e$, and therefore $\varphi(\bs)=\xi^{k}(\bs)$ with $k=t+p+v$.
\end{itemize}
For an interpretation of the integer $t$, see Remark \ref{remisoPsi}.
\end{proof}

\begin{exa} As in Example \ref{exapeel}, take $e=3$, $\bs=(3,3,4)$ and  $\bla=(3^2.2,2^2.1,3.1^2)$.
Then $\xi^3(\bs)=(0,0,1)\in\sD_e$. Hence $\varphi(\bs)=(0,0,1)$.
\end{exa}

\medskip

\subsection{Determining the FLOTW multipartition}

\label{isopsi}

In order to compute the multipartition $\varphi(\bla)$, we need to introduce a new crystal isomorphism, which acts on cylindric multipartitions.
In fact, the only difference between a cylindric multipartition and a FLOTW multipartition is the possible presence of pseudoperiods.
Therefore, we want to determine an isomorphism which maps a cylindric multipartition to another cylindric multipartition with one less pseudoperiod.
Applying this recursively, we will eventually end up with a FLOTW multipartition equivalent to $\bla$.

\medskip

Take $\bs\in\sS_e$ and $\bla\in\sC_\bs$.
Let $\al$ be the width of $P(\bla)$, the first pseudoperiod of $\bla$.

Denote by $\psi(\bla)$ the multipartition $\bmu$ charged by $\xi(\bs)$ defined as follows:
\begin{itemize}
\item $\bmu^c$ contains all parts $\la^c_a$ of $\la^c$ such that $1\leq\la^c_a<\al$, for $c\in\llb 1 , l \rrb$,
\item $\bmu^c$ contains all parts $\la^{c-1}_a$ of $\la^{c-1}$ such that $\la^{c-1}_a>\al$ for $c>1$, 
and $\bmu^1$ contains all parts $\la^{l}_a$ of $\la^{l}$ such that $\la^{l}_a>\al$,
\item $\bmu^c$ contains all parts $\la_a^c=\al$ whose rightmost node does not belong to $P(\bla)$, for $c\in\llb 1 , l \rrb$.
\end{itemize}

This naturally defines a mapping $\ga \longleftrightarrow \Ga$ from the set of nodes of $\bla\backslash P(\bla)$ (see Definition \ref{defpp})
onto the set of nodes of $\psi(\bla)$.
We then say that $\Ga$ is \textit{canonically associated to} $\ga$, and conversely.

\begin{exa} \label{exapsi}
Let us go back to Example \ref{exapp2}. We had $e=4$, $\bs=(5,6,8)$ and
$$\bla=\left( \; \young(56789\dix,456789,34,2,1) \; , \; \young(678,5\bsix,4\bcinq,34,2,1) \; , 
\; \young(89\dix\onze\douze\treize,7\bhuit,6\bsept,5,4,3) \; \right),$$
where the bold contents correspond to the first pseudoperiod (whose width is $\al=2$).

Then $\xi(\bs)=(4,5,6)$, and 
$$\psi(\bla)= \left( \; \young(456789,34,2,1) \; , \; \young(56789\dix,456789,34,2,1) \; , \; \young(678,5,4,3) \; \right)=(6.2.1^2,6^2.2.1^2,3.1^3)$$
\end{exa}

% We make the following observation on the contents of $\psi(\bla)$ once charged by $\xi(\bs)$:
% 
% \begin{enumerate}
% \item The contents of the nodes of the shifted parts (i.e. greater than $\al$) are unchanged
% if the parts are in a partition $\la^c$ with $c<l$; and are translated by $-e$ if the parts are in the partition $\la^l$.
% \item The contents of the nodes of the parts that do not move (i.e. smaller than or equal to $\al$) are unchanged.
% \end{enumerate}

We observe that the contents of canonically associated nodes are unchanged, 
except for the nodes of $\bla$ that lie in part of $\la^l$ greater than $\al$, whose content is translated by $-e$.
More formally, this writes:

\begin{prop}\label{propcontents} Denote $\psi(\bla)=:\bmu$.
% $=(\mu^1,\dots,\mu^l)$.

\begin{enumerate}
 \item Let $\Ga=(A,B,C)$ be a node of $\bmu$ with $\mu_A^C>\al$. Denote by $\ga$ the node of $\bla$
canonically associated to $\ga$.

\begin{itemize}
          \item If $c>1$, then $\cont_{\bmu}(\Ga)=\cont_{\bla}(\ga)$.
% If $1<c\leq l$, for all $a\in\llb 1, \h(\mu^c) \rrb$ such that $\mu^c_a>\al$ and for all $i\in\llb 1, \mu^c_a \rrb$, 
% $$\cont_{\bmu}(a,i,c)=\cont_{\bla}(a,i,c-1).$$
         \item If $c=1$, then $\cont_{\bmu}(\Ga)=\cont_{\bla}(\ga)-e$.
% For all $a\in\llb 1, \h(\mu^1) \rrb$ such that $\mu^1_a>\al$ and for all $i\in\llb 1, \mu^1_a \rrb$, 
% $$\cont_{\bmu}(a,i,1)=\cont_{\bla}(a,i,l)-e.$$
        \end{itemize}
 \item Let $\Ga=(A,B,C)$ be a node of $\bmu$ with $\mu_A^C\leq\al$. Denote by $\ga$ its canonically associated node in $\bla$.
Then $\cont_{\bmu}(\Ga)=\cont_{\bla}(\ga)$.
% For all $c\in\llb 1, l\rrb$, for all $a\in\llb 1, \h(\mu^c) \rrb$ such that $\mu^c_a\leq\al$ and for all $i\in\llb 1, \mu^c_a \rrb$,
% $$\cont_{\bmu}(a,i,c) = \cont_{\bla}(a,i,c).$$ 
\end{enumerate}
\end{prop}

\begin{notation}\label{notation0} Let $\bs$ be a $l$-charge in $\sS_e$, $\bla\in\sC_\bs$ non FLOTW, and $c\in\llb 1 ,l \rrb$.
Set $\al$ to be the width of the first pseudoperiod of $\bla$.
We denote:
\begin{itemize}
 \item $N^{>\al}_{c}$ the number of parts greater than $\al$ in $\la^{c}$,
\item $N^{\al}_c$ the number of parts equal to $\al$ in $\la^c$
that are deleted in $\bla$ to get  $\psi(\bla)$ (i.e. parts whose rightmost node belongs to the pseudoperiod).
\end{itemize}
\end{notation}

\begin{proof}
Of course, it is sufficient to prove this for only one node in each part considered, since the contents of all other nodes of the part are then determined.
We prove it only for the leftmost nodes, i.e. the ones of the form $(A,1,C)$.
\begin{enumerate}
 \item This is clear since the multicharge associated to $\bmu$ is simply the cyclage of $\bs$ (that shifts $\bs$ "to the right" and maps $s_1$ to $s_l-e$), 
and the parts greater than $\al$ are similarly shifted in $\bla$ to get $\bmu$.
 \item Let $\Ga=(A,1,C)$ be a node in $\bmu$ such that $\mu^C_A\leq\al$, so that $\ga=(a,1,c)$ with $c=C$. 

First, assume $C>1$.
Then $$\cont_{\bmu}(\Ga) = s_{C-1} - N^{>\al}_{C-1}.$$
On the other hand, $$\cont_{\bla}(\ga) = s_C - N^{>\al}_{C} - N^{\al}_C.$$
Now by definition of $P(\bla)$, which is charged by $\xi(\bs)$, we have $$s_C-s_{C-1}=N_C^{>\al}+N_C^\al-N_{c-1}^{>\al},$$
which is equivalent to
$$N^\al_C = s_C - N^{>\al}_C - s_{C-1} - N^{>\al}_{C-1}.$$
Hence we have 
$$\begin{array}{ccl}
\cont_{\bla}(\ga) & = & s_C - N^{>\al}_{C} - N^{\al}_C \\
                  & = & s_C - N^{>\al}_{C} - ( s_C - N^{>\al}_C - (s_{C-1} - N^{>\al}_{C-1} )) \\
                 &  = & s_{C-1} - N^{>\al}_{C-1} \\
                  & = &  \cont_{\bmu}(\Ga).
\end{array}$$

\medskip

Now, assume $C=1$. The argument is the same:
$$\cont_{\bmu}(\Ga) = s_{l} - e - N^{>\al}_{l}, \mand$$
$$\cont_{\bla}(\ga) = s_1 - N^{>\al}_{1} - N^{\al}_1.$$
Moreover, we have:
$$\begin{array}{ccl} N^{\al}_1 & = & e - \sum_{d=2}^l N_d^{\al} \\
                            & = & e - \sum_{d=2}^l (r_d-r_{d-1}) \\
                            & = & e - r_l + r_1 \\
                            & = & e - (s_l - N^{>\al}_l) + (s_1 - N^{>\al}_1) \text{ \quad using the above case},   
  \end{array} $$

which implies that 

$$\begin{array}{ccl}
  \cont_{\bmu}(\Ga) & = & s_{l} - e - N^{>\al}_{l} \\
                    & = & s_1 - N^{>\al}_1 - e + s_l -  N^{>\al}_l - s_1 + N^{>\al}_1 \\
                    & = & s_l - e - N^{>\al}_l \\
                    & =  & \cont_{\bla}(\ga).
  \end{array}$$

\end{enumerate}

\end{proof}

We could also have chosen to describe this procedure directly on symbols.
In fact, the first pseudoperiod in the symbol of a cylindric multipartition would be the sequence of $e$ consecutive entries verifying the following properties:
\begin{enumerate}
 \item they lie in consecutive columns,
 \item if there are several candidates for an entry, choose the one lying in the lowest row,
 \item the first entry of the pseudoperiod is maximal amongst all candidates.
\end{enumerate}
Clearly, this makes this pseudoperiod unique.
%Then, one also defines recursively the $k$-th pseudoperiod. 
Note that here we only consider the entries encoding non-empty parts.

Denote by $P(\bla)=(\be,\be-1,\dots,\be-e+1)$ the pseudoperiod in a symbol $\fB_\bs(\bla)$.
The map $\psi$ then acts on $\fB_\bs(\bla)$ by:
\begin{itemize}
 \item keeping all entries smaller than $\be$ which do not belong to $P(\bla)$ in the same place,
 \item shifting the entries greater than $\be-1$ which do not belong to $P(\bla)$ to the row just above (and translating by $-e$ and shifting from the top to bottom row if need be),
 \item deleting the entries of $P(\bla)$.
\end{itemize}

\begin{rem}\label{remabacus}
If we replace symbols by abacuses (which are almost identical representations of charged multipartition, in the sense that they bring out the beta numbers)
note that this resembles Tingley's tightening procedure \cite{Tingley2008} on descending abacuses.
Here, however, we do not require to shift one entry per row at a time. Besides, the heart of our procedure lies in the deletion of entire pseudoperiods, 
whereas Tingley's algorithm is rather based upon the shifting phenomenon (yet enabling the deletion of exactly $e$ nodes in the multipartition).
\end{rem}

\medskip

We now aim to prove that the map $\psi$ we have just defined is in fact a crystal isomorphism between
connected components of Fock spaces crystals (this is upcoming Theorem \ref{thmisopsi}).
In order to do that, we will need the following three lemmas, in which we investigate the compatibility
between $\psi$ and the possible actions of the crystal operators $\tf_i$.
For the sake of clarity (the proofs of these statements being rather technical), they are proved in Appendix \ref{appendix}.

\begin{lem}\label{lemisopsi1}
Suppose that $\ga^+=(a,\al+1,c)$ is the good addable $i$-node of $\bla$, with $\ga\in P(\bla)$.
Then
\begin{itemize}
\item $\De^+=(a,\al+1,c+1)$ is the good addable $i$-node of $\psi(\bla)$ if $1\leq c < l$,
\item $\De^+=(a,\al+1,1)$ is the good addable $i$-node of $\psi(\bla)$ if $c=l$.
\end{itemize}
\end{lem}

\begin{lem}\label{lemisopsi2}
Suppose that $\ga^+=(a,\la^c_a +1,c)$ is the good addable $i$-node of $\bla$, with $\la^c_a <\al$ or
[$\la^c_a=\al$ and $\ga\notin P(\bla)$].
Then $\Ga^+=(a-D,\la^c_a+1,c)$ is the good addable $i$-node of $\psi(\bla)$, where
\begin{itemize}
\item $D=N^{>\al}_c-N^{>\al}_{c-1} + N^\al_c-N^\al_{c-1}$ if $c\geq1$,
\item $D=N^{>\al}_1-N^{>\al}_{l} + N^\al_1-N^\al_{l}$ if $c=1$ (see Notation \ref{notation0}).
\end{itemize}
\end{lem}

\begin{lem}\label{lemisopsi3}
Suppose that $\ga^+=(a,\la^c_a +1,c)$ is the good addable $i$-node of $\bla$, with $\la^c_a >\al$. 
Then 
\begin{itemize}  
\item $\Ga^+=(a,\la^c_a +1,c+1)$ is the good addable $i$-node of $\psi(\bla)$ if $1\leq c < l$,
\item $\Ga^+=(a,\la^c_a +1,1)$ is the good addable $i$-node of $\psi(\bla)$ if $c=l$.
\end{itemize}
\end{lem}

We are now ready to prove the following key result.

\begin{thm}\label{thmisopsi} Let $|\bladot,\bs\rangle$ be a cylindric $l$-partition.
 The map $$\begin{array}{cccc}  \psi : & B(\bladot,\bs) & \lra & B(\overset{\bullet}{\psi(\bla)},\xi(\bs)) \\
                                       & \bla  & \longmapsto & \psi(\bla) 
          \end{array}$$
is a crystal isomorphism. We call it the \textit{reduction} isomorphism for cylindric multipartitions.
\end{thm}

\begin{proof}
We need to prove that for all $i\in\llb 0, e-1 \rrb $, 
\begin{equation}\label{iso} \tf_i(\psi(\bla))=\psi(\tf_i(\bla)). \end{equation}

Thanks to the previous lemmas, we know precisely what $\tf_i(\psi(\bla))$ is.
It remains to understand the right hand side of (\ref{iso}), by looking at the pseudoperiod of $\tf_i(\bla)$.
Let $P(\bla)=(\ga_1,\ga_2,\dots,\ga_e)$. 

The operator $\tf_i$ acts on $\bla$ either by:
\begin{enumerate}
\item Adding a node to a part $\al$ whose rightmost node belongs to $P(\bla)$. This is the setting of Lemma \ref{lemisopsi1}.
Let $\ga=\ga_k=(a,\al,c)$ be the node of $P(\bla)$ such that $\ga^+$ is the good addable $i$-node of $\bla$.
In this case, we have $$P(\psi(\bla))=(\ga_1,\ga_2,\dots,\ga_{k-1},\de,\ga_{k+1},\dots,\ga_e),$$
where: 
\begin{itemize}
 \item $\de=(b,\al,c+1)$, with $b=a+N_{c+1}^{>\al}+N_{c+1}^\al-N_c^{>\al}$, if $c<l$, and
\item $\de=(b,\al,1)$, with $b=a+N_{1}^{>\al}+N_{1}^\al-N_l^{>\al}$, if $c=l$.
\end{itemize}
Indeed, this node $\de$ is the same as the one determined in the proof of Lemma \ref{lemisopsi1} (and whose canonically associated node is $\De$).
The value of the row $b$ is simply computed using the fact that:
\begin{enumerate}
\item there is no part $\al$ above the part of rightmost node $\ga$,
\item all parts $\al$ above the part of rightmost node $\de$ in $\tf_i(\bla)$ have an element of $P(\bla)$ as rightmost node.
\end{enumerate}

But the part of $\tf_i(\bla)$ whose rightmost node is $\ga^+$ is a part of size greater than $\al$, 
and is is therefore shifted to the $(c+1)$-th component (if $c<l$), or the first component (if $c=l$) when building $\psi(\tf_i(\bla))$.
Hence, by deleting the elements of $P(\psi(\bla))$ and shifting the parts greater than $\al$,
we end up with the same multipartition as $\tf_i(\psi(\bla))$, whence the identity $\tf_i(\psi(\bla))=\psi(\tf_i(\bla))$.
This is illustrated in Example \ref{exaisopsi} below.

\item Adding a node to a part $\leq\al$ whose rightmost node does not belong to $P(\bla)$. This is the setting of Lemma \ref{lemisopsi2}.
In this case , we have $P(\psi(\bla))=P(\bla)$.
It is then straightforward that $\tf_i(\psi(\bla))=\psi(\tf_i(\bla))$.

\item Adding a node to a part $>\al$ (whose rightmost node necessarily does not belong to $P(\bla)$). This is the setting of Lemma \ref{lemisopsi3}.
Here, we also have $P(\psi(\bla))=P(\bla)$, as in the previous point.

\end{enumerate}

\end{proof}

\begin{exa}\label{exaisopsi}
We take the same example as \ref{exalemisopsi1}, 5., namely $\bla=(3.2.1^2,4.2.1,2^3)$, $\bs=(2,3,4)$, $e=4$ and $i=0$.
Then we have the following constructions:

$$ \xymatrix@!0 @R=5cm @C=7cm{
  \left( \; \young(234,1\bdeux,0,\moinsun) \; , \; \young(3456,2\btrois,1) \; ,\; \young(4\bcinq,3\bquatre,23) \; \right) \quad
  \ar[r]^-{\psi}\ar[d]_-{\tf_0}
&
  \quad \left( \; \young(0,\moinsun) \; , \; \young(234,1) \; ,\; \young(3456,23) \; \right)
  \ar[d]^-{\tf_0} 
\\
    \left( \; \young(234,1\bdeux,0,\moinsun) \; , \; \young(3456,234,1) \; ,\; \young(4\bcinq,3\bquatre,2\btrois) \; \right) \quad
\ar[r]^-{\psi}
&
    \quad \left( \; \young(0,\moinsun) \; , \; \young(234,1) \; ,\; \young(3456,234) \; \right)
  }$$

The bold contents represent the pseudoperiods. This illustrates the commutation between the operators $\psi$ and $\tf_i$.

\end{exa}

\begin{rem}\label{remisopsi}
Note that the charged multipartition $|\psi(\bla) , \xi(\bs) \rangle $ that we get is not cylindric anymore in general,
because $\xi(\bs)$ might not be in $\sS_e$.
If it is, then it is clear that $\psi(\bla)$ has one pseudoperiod less that $\bla$.
\end{rem}

\begin{rem}\label{remisopsi2}
Interestingly, this isomorphism $\psi$ can be seen as a generalisation of the cyclage isomorphism $\xi$.
Indeed, $\xi$ would be the version of $\psi$ for pseudoperiods of width $0$ 
(which can be found in any multipartition, considering that they have infinitely many parts of size $0$).
\end{rem}

By a simple use of the cyclage isomorphism, we can now easily determine a refinement $\Psi$ of the reduction isomorphism $\psi$
which maps a cylindric multipartition to another cylindric multipartition with one pseudoperiod less.

\medskip

Let $\bla\in\sC_\bs$. 
If $\xi(\bs)\notin\sS_e$, denote by $p$ the number of components of $\xi(\bs)$ equal to $\xi(\bs)_l$.
By Remark \ref{remmcpeel} and the proof of Proposition \ref{propmc}, $\xi^{1+p}(\bs) \in\sS_e$.
% 

% \medskip

% We define  $$\varphi(\bla):= (\xi^{k-t}\circ\psi^t)(\bla),$$ where $t$ is the number of pseudoperiods in $\bla$.

% \begin{prop}\label{propisophi}
% For all $\bla\in\sC_\bs$, we have$\varphi(\bla)\in\sC_{\varphi(\bs)}$.
% Moreover, the map  $\varphi : |\bla,\bs\rangle  \longmapsto   |\varphi(\bla),\varphi(\bs) \rangle$ 
% is the canonical $\Ueprime$-crystal isomorphism.
% \end{prop}

Define $\Psi(\bla)$ and $\Psi(\bs)$ in the following way:

\begin{itemize}
 \item If $\xi(\bs)\in\sS_e$, then $\Psi(\bla):=\psi(\bla)$ and $\Psi(\bs):=\xi(\bs)$
 \item If $\xi(\bs)\notin \sS_e$, then $\Psi(\bla):=(\xi^p\circ \psi)(\bla)$ and $\Psi(\bs):=\xi^{1+p}(\bs)$.
\end{itemize}

We denote $\Psi : |\bla,\bs\rangle \longmapsto  |\Psi(\bla),\Psi(\bs) \rangle$.
Then by construction, the following result holds:

\begin{prop}\label{propisoPsi}
 For all $\bla\in\sC_\bs$, we have $\Psi(\bla)\in\sC_{\Psi(\bs)}$.
 Moreover, $\Psi$ is a $\Ueprime$-crystal isomorphism, and $\Psi(\bla)$ has one pseudoperiod less than $\bla$.
\end{prop}

We can now determine the canonical crystal isomorphism $\varphi$ for cylindric multipartitions.
Recall that we have already determined $\varphi(\bs)$ in Proposition \ref{propmc}.
It writes $\varphi(\bs)=\xi^k(\bs)$ for some $k$ explicitely determined.

\begin{rem}\label{remisoPsi}
The integer $t$ defined in the alternative proof of Proposition \ref{propmc} (end of Paragraph \ref{mc}) is simply the number of pseudoperiods in $\bla$.
\end{rem}

Denote $t$ the number of pseudoperiods in $\bla$.
Applying $t$ times $\Psi$ to $\bla$, we end up with a FLOTW multipartition, but charged by an element $\Psi^t(\bs)$ which might not be in $\sD_e$.
We now simply need to adjust it by some iterations of the cyclage isomorphism $\xi$.
Since $\Psi^t(\bs) \in\sS_e$, we are ensured that $\varphi(\bs)=(\xi^{u}\circ \Psi^t)(\bs)$ for some $u\in\Z$ easily computable.

Hence, we set $$\varphi(\bla):=(\xi^{u}\circ \Psi^t)(\bla),$$ and the following theorem is straightforward.

\begin{thm}\label{thmisophi} Let $|\bladot,\bs\rangle$ be a cylindric $l$-partition.
 The map $$\begin{array}{cccc} \varphi :&  B(\bladot,\bs) & \lra & B(\overset{\bullet}{\varphi(\bla)},\varphi(\bs)) \\
                                    & \bla  & \longmapsto  &  \varphi(\bla)
\end{array}$$ 
is the canonical $\Ueprime$-crystal isomorphism for cylindric multipartitions.
\end{thm}

\begin{rem}\label{remisophi}
Note that to determine $\varphi(\bs)$, we could also have built first $|\Psi^t(\bla),\Psi^t(\bs)\rangle$, and found $u$ such that $\xi^{u}(\Psi^t(\bs))\in\sD_e$.
Then, we would have set $\varphi(\bs)= \xi^{u}(\Psi^t(\bs))$. 
Clearly, this construction would give the same multicharge as the construction of $\varphi(\bs)$ in Section \ref{mc}.
The point of Section \ref{mc} is to show that the suitable multicharge is directly computable using only cyclages of $\bs$.
\end{rem}

\medskip

We can therefore express the canonical crystal isomorphism $\Phi$ in full generality.
Starting from any charged multipartition $|\bla,\bs\rangle$, we first apply $\rs$, we get $\bla'=\rs(\bla)$. 
Then, according to Section \ref{algorithm}, there is an integer $m$ such that $(\rs\circ\xi)^m(\bla')$ is cylindric.
Finally, we use $\Psi^t$ to delete all pseudoperiods
and adjust everything using $\xi^u$ so that we end up in the fundamental domain $\sD_e$ (Theorem \ref{thmisophi}).
It is easy to see that $t$ does not depend on $\bla$ and is constant along the crystal, since
it is just the number of pseudoperiods of $\bla$.
Moreover, $m$ does not depend on $\bla$ either because of Proposition \ref{mconstant}. 
Hence, neither does $u$, since it is just the "adjusting" multiplicity of $\xi$.
This gives a generic expression for $\Phi$, regardless of the multipartition $\bla$ we start with, and depending only of
the connected component $B(\bladot,\bs)$.
With the notation $\varphi=\xi^u \circ \Psi ^t$, this gives the following result:

\begin{cor}\label{corcci} Let $\bs\in\Z^l$ and $B(\bladot,\bs)$ be a connected component of the crystal graph of $\cF_\bs$.
Then the canonical crystal isomorphism is $$\Phi = \varphi \circ (\rs\circ\xi)^m\circ \rs. $$
\end{cor}

\begin{rem}\label{remcci}
In concrete terms, since the data of $B(\bladot,\bs)$ is given by the $l$-charge $\bs$ and some vertex $\bla \in B(\bladot,\bs)$,
one can determine the expression of $\Phi$ by making these manipulations on $|\bla,\bs\rangle$.
If we know the highest weight vertex $\bladot$, it is natural to take $\bla=\bladot$ to compute $\Phi$.
\end{rem}

\subsection{Analogy with finite type $A$}

\medskip

With this purely algorithmic expression for $\Phi$, one can draw a parallel with the non affine case.
Recall that in Section \ref{e=inf}, we have seen that the canonical crystal isomorphism in the finite case was precisely Schensted's insertion procedure (Corollary \ref{corrs}).
In this perspective, one can naturally regard the canonical crystal isomorphism $\Phi$ as an analogue of Schensted's algorithm for affine type $A$.
Note that in Remark \ref{remanaloguersk}, we extend this comparison to the whole Robinson-Schensted-Knuth correspondence. 

Besides, in finite type $A$, it is well known that the equivalence relation of Definition \ref{defequivmp} reduces to the elementary \textit{Knuth relations},
which give rise to the classic structure of \textit{plactic monoid}, see for instance \cite{Lothaire2002} and \cite{Lecouvey2007}.
It would be interesting to see if it possible to reduce this equivalence relation in the affine case in a similar fashion, 
and describe it in terms of (even more) elementary crystal isomorphisms.

\section{An application}

\label{hw}

Let $\bs\in\Z^l$.
In this last section, we deduce a direct characterisation of all the vertices of any connected component of $B(\cF_\bs)$.

Fix $B(\bladot,\bs)$ a connected component of $B(\cF_\bs)$.
This implies that we know one of the vertices of $B(\bladot,\bs)$.
We assume without loss of generality (see Remark \ref{remcci}) that we know $\bladot$.
Then the expression of the canonical crystal isomorphism  $$\begin{array}{cccc} \Phi : & B(\bladot,\bs) & \lra & B(\br) \\ 
 & |\bla,\bs\rangle  & \longmapsto & |\bmu,\br\rangle \end{array}$$ is obtained by manipulating $|\bladot,\bs\rangle$.	
Corollary \ref{corcci} shows the three basic crystal isomorphisms needed to construct $\Phi$,
namely $\rs$, $\xi$, and $\psi$ (keeping in mind that $\varphi= \xi^k\circ\psi^t$ for some $t$ and $k$).
Amongst them, $\xi$ is the only map which is clearly invertible.
However, it is possible, keeping extra information, to make $\rs$ and $\psi$ invertible.

\subsection{Invertibility of the crystal isomorphism $\rs$}

First, it is well known (e.g. \cite{Fulton1997}) that the correspondence $\reading(\bla,\bs) \longmapsto \sP(\reading(\bla,\bs))$
becomes a bijection if we also associate to $\reading(\bla,\bs)$ its "recording symbol", i.e. the symbol with the same shape as $\sP(\reading(\bla,\bs))=:\sP$
in which we put the entry $k$ in the spot where a letter appears at the $k$-th step.
We denote by $\sQ(\reading(\bla,\bs))$ or simply $\sQ$ this symbol.

\begin{exa}\label{exainvertiblityrs} Take $\bs=(2,0)$ and $\bla=(3.2,3^2)$.
Then $$\fB_\bs(\bla)= 
\begin{pmatrix}
 0 & 4 & 5 & & \\
 0 & 1 & 2 & 5 & 7
\end{pmatrix}.$$

We get $\reading(\bla,\bs)= 54075210$.
We can thus give the sequence of symbols leading to $\sP$, and, on the right, the corresponding recording symbols, leading to $\sQ$.

$$\begin{array}{clcccccl}
&\begin{pmatrix}
5
\end{pmatrix} 
& & & &&
&\begin{pmatrix}
 1
\end{pmatrix}
\\
& & &&&
\\
&\begin{pmatrix}
 4 & 5
\end{pmatrix}
& & &&&
&\begin{pmatrix}
 1 & 2
\end{pmatrix}
\\
& & &&&
\\
&\begin{pmatrix}
 0 & 4 & 5
\end{pmatrix}
& & &&&
&\begin{pmatrix}
 1 &2 & 3
\end{pmatrix}
\\
& &&& &
\\
&\begin{pmatrix}
 0 & 4 & 5 \\
 7 & &
\end{pmatrix}
& & &&&
&\begin{pmatrix}
1 & 2 & 3 \\
4 & & 
\end{pmatrix}
\\
&&&&&
\\
&\begin{pmatrix}
 0 & 4 & 5 \\
 5 & 7 & 
\end{pmatrix}
&&&&&
&\begin{pmatrix}
 1 & 2 & 3 \\
4 & 5 & 
\end{pmatrix}
\\
&&&&&
\\
&\begin{pmatrix}
 0 & 4 & 5 \\
2 & 5 & 7
\end{pmatrix}
&&&&& 
&\begin{pmatrix}
 1 & 2 & 3 \\
4 &5 & 6
\end{pmatrix}
\\
&&&&&
\\
&\begin{pmatrix}
 0 & 2 & 4 & 5 \\
1 & 5 & 7 &
\end{pmatrix}
&&&&&
&\begin{pmatrix}
 1 & 2 & 3 & 7\\
4 & 5 & 6
\end{pmatrix}
\\
&&&&&
\\
\sP= &
\begin{pmatrix}
 0 & 1 & 2 & 4 & 5 \\
0 & 5 & 7 & &
\end{pmatrix}
&&&&&
\sQ= &
\begin{pmatrix}
1 & 2 & 3 & 7 & 8 \\
4 & 5 & 6 & &
\end{pmatrix}
  \end{array}$$

\end{exa}

\medskip

Therefore, this extra data $\sQ$ turns $\rs$ into a bijection.

\subsection{Invertibility of the reduction isomorphism}

Recall that the $l$-tuple charging  $\psi(\bla)$ is nothing but $\xi(\bs)$.
Hence, the computation of $\bs$ from $\psi(\bs)=\xi(\bs)$ is straightforward.
Moreover, starting from $\psi(\bla)$, it is easy to recover the $l$-partition $\bla$
provided we know the width $\al$ of the pseudoperiod that has been deleted.
In fact:
\begin{enumerate}
 \item  This data determines which parts will stay in the same component (namely the ones smaller than or equal to $\al$),
and which will be shifted "to the left" (namely the ones greater than $\al$).
Moreover, the property on the contents (Proposition \ref{propcontents}), which says that all nodes must keep the same content,
ensures that we can keep the boxes filled in with the same integers.
\item It remains to insert the $e$ parts of the $\al$-pseudoperiod at the right locations, i.e. so that the diagram
obtained is in fact the Young diagram of a charged multipartition
(which is possible thanks, again, to Proposition \ref{propcontents}).
\end{enumerate}

\begin{exa} We take, as in Example \ref{exapsi}, $e=4$ and
$$\psi(\bla)= \left( \; \young(456789,34,2,1) \; , \; \young(56789\dix,456789,34,2,1) \; , \; \young(678,5,4,3) \; \right),$$
and we suppose that we know the width of $P(\bla)$, namely $\al=2$.
Then the two steps above give:

\begin{enumerate}
 \item Shifting the parts greater that $\al$ and keeping the same filling:
$$\left( \; \young(56789\dix,456789,34,2,1) \; , \; \young(678,34,2,1) \; , \; \young(89\dix\onze\douze\treize,5,4,3) \; \right).$$
Note that at this point, this object cannot be seen as a charged $l$-partition (the entries in the boxes are not proper contents).
\item Inserting coherently the $4$ missing parts of size $2$ (represented in bold type):
$$\left( \; \young(56789\dix,456789,34,2,1) \; , \; \young(678,\bcinq\bsix,\bquatre\bcinq,34,2,1) \; , 
\; \young(89\dix\onze\douze\treize,\bsept\bhuit,\bsix\bsept,5,4,3) \; \right),$$
which is indeed equal to $\bla$.
\end{enumerate}

\end{exa}

\medskip

Hence, this extra data $\al$ turns $\psi$ into an invertible map.

\subsection{A pathfinding-free characterisation of the vertices of any connected component}

Both $\rs$ and $\psi$ being turned into bijections with the appropriate extra data,
we can turn the canonical crystal isomorphism $\Phi : B(\bladot,\bs) \lra B(\br)$ into an invertible map.
Concretely, this is achieved by collecting the recording data when computing $\Phi$ by manipulating $|\bladot,\bs\rangle$.
According to Corollary \ref{corcci}, this recording data consists of a pair $(\underline{\sQ},\underline{\al})$, where 
\begin{itemize}
\item $\underline{\sQ}$ is an $(m+1)$-tuple of recording symbols $(\sQ_0,\sQ_1,\dots ,\sQ_m)$ (corresponding to the occurences of $\rs$ in $\Phi$), and
 \item $\underline{\al}$ is a $t$-tuple of integers $(\al_1,\dots, \al_t)$ (the widths of the different pseudoperiods),
\end{itemize}
where $t$ and $m$ are such that $\Phi = \xi^u \circ \Psi^t \circ (\rs\circ\xi)^m\circ \rs$.

We write $\Phi^{-1} : B(\br) \lra B(\bladot,\bs)$ for the inverse map.

\medskip

Now, since the vertices of $B(\br)$ are FLOTW $l$-partitions, they have an explicit, non-recursive characterisation.
Hence, we have the following direct characterisation of all the vertices of $B(\bladot,\bs)$:

\begin{thm}\label{thmnonrec}
The set of vertices of $B(\bladot,\bs)$ is equal to $$\Big\{ \Phi^{-1}(\bmu) \; ; \; \bmu \in B(\br) \Big\}.$$
% where $\Phi$ and $\Phi^{-1}$ are completely determined by the highest weight vertex $|\bladot,\bs\rangle$.
\end{thm}

\begin{rem}
Following Remark \ref{remcci}, if we do not know $\bladot$ but some other $\bla\in B(\bladot, \bs)$ instead, one can still determine $\Phi$ and $\Phi^{-1}$.
Then Theorem \ref{thmnonrec} enables the direct computation of the highest weight vertex, namely $\bladot = \Phi^{-1}(\bemptyset)$.
\end{rem}

\begin{rem}\label{remanaloguersk}
This gives an analogue of the Robinson-Schensted-Knuth correspondence \cite{Lothaire2002}: we have a one-to-one correspondence
$$ |\bla,\bs\rangle \overset{1-1}{\longleftrightarrow} (|\bmu,\br \rangle , (\underline{\sQ},\underline{\al}) )$$
between the set of charged $l$-partitions on the one hand, and the set of pairs consisting of
an FLOTW $l$-partition $|\bmu,\br\rangle$ and a recording data $(\underline{\sQ},\underline{\al})$ on the other hand.

Alternatively, $(\underline{\sQ},\underline{\al})$ can be replaced by $|\bladot,\bs\rangle$, 
since the recording data is entirely determined by the highest weight vertex.
\end{rem}

\bigskip

\appendix
\section{Proof of Lemmas \ref{lemisopsi1} , \ref{lemisopsi2} and \ref{lemisopsi3}} \label{appendix}

\textit{Proof of Lemma \ref{lemisopsi1}.}
The proof first splits in three cases (which, in turn split in subcases), even though ultimately, the argument is the same. 
In each case (and subcase), we determine a certain node $\de$ of $\bla$ which is not in $P(\bla)$. 
Then, we show that the node of $\psi(\bla)$ canonically associated to $\de$ is the node $\De$ we expect.

\begin{enumerate}
\item Assume first that $\ga$ is the first element of $P(\bla)$.
Denote by $\ga_e=(a_e,\al,c_e)$ the last node of $P(\bla)$. Then $\cont(\ga_e)=\cont(\ga^+)-e$, and since $\ga^+$ is an $i$-node, then so is $\ga_e$.

\begin{enumerate}

\item Suppose that $\ga_e$ is removable (cf Example \ref{exalemisopsi1}, 1.). 
Since $\ga^+$ is the good addable $i$-node, the letter $R$ produced by $\ga_e$ in the $i$-word must not
simplify with the $A$ produced by $\ga^+$. This means that there exists an addable node $\tga^+$ in $\la^{\tc}$ with either 
\begin{itemize}
\item $\tc>c$ and $\cont(\tga^+)=\cont(\ga^+)$, or
\item $\tc < c_e$ and $\cont(\tga^+)=\cont(\ga^+)-e$ (one cannot have $\tc=c_e$ since otherwise $\ga_e$ would not be removable).
\end{itemize}

In the first case, this means that there is an integer $\be$ in the $\tc$-th row of $\fB_\bs(\bla)$
which is also in the $c$-th row and the same column. By the semistandardness of $\fB_\bs(\bla)$ (because $|\bla,\bs\rangle$ is cylindric),
we are ensured that $\be$ is also present in the row $d$ and the same column for all $d\in\llb c , \tc \rrb$.
This is equivalent to saying that there is a part of size $\al$ in each component $\la^d$ with $d\in\llb c , \tc \rrb$ 
whose rightmost node $\hga$ verifies $\cont(\hga)=\cont(\ga)$. We denote by $\de$ the one located in the component $\la^{c+1}$.
In particular, $\de^+$ is an $i$-node.
Moreover, $\de^+$ is an addable node of $\bla$. Indeed, if it is not, then there is a part of size $\al$ just above the part
whose rightmost node is $\tga$. Denote by $\overline{\ga}$ its rightmost node. Then $\cont(\overline{\ga})=\cont(\de)+1=\cont(\ga)+1$.
This implies that $\overline{\ga}$ must be in $P(\bla)$, which contradicts the fact that $\ga$ is the first element of $P(\bla)$.

In the second case, there is an integer $\be$ in the $\tc$-th row of $\fB_\bs(\bla)$
which is also in the $c_e$-th row and the same column, say column $k$. Again, since the symbol is semistandard, the elements $\fb_{d,k}$
appearing in row $d$, with $d\in\llb 1, \tc \rrb$, and in column $k$ verify  $\fb_d\geq\fb_{d'}\geq\be$ for $1\leq d'<d < c_e$.
But the cylindricity property also implies that $\fb_{l,k+e}\geq \fb_{1,k}+e$.
Actually, the $(k+e)$-th column of the symbol is also the column that contains the integer corresponding to the first node of $P(\bla)$, since pseudoperiods have length $e$.
Moreover, this element is equal to $\be+e$. In other terms, $\fb_{c,k+e}=\be+e$.
By semistandardness again, one must have $\be+e\geq \fb_{d,k+e} \geq \fb_{d',k+e}$
for all $c\leq d < d' \leq l$.
To sum up, we have $$ \be+e = \fb_{c,k+e} \geq \fb_{c+1,k+e} \geq \dots \geq \fb_{l,k+e} \geq \fb_{1,k}+e \geq \dots \geq\fb_{c_e-1,k}+e \geq \fb_{c_e,k}+e = \be +e ,$$
thus all inequalities are in fact equalities. 
As in the first case, this means in particular that there is a part of size $\al$ in $\la^{c+1}$ (if $c<l$) or in $\la^{1}$ (if $c=l$) 
whose rightmost node $\de$ verifies $\cont(\de)=\cont(\ga)-e$. In particular, $\de^+$ is an $i$-node.
Moreover, $\de^+$ is an addable node of $\bla$. Indeed, if it is not, then there is a part of size $\al$ just above the part
whose rightmost node is $\tga$. Denote by $\overline{\ga}$ its rightmost node. Then $\cont(\overline{\ga})=\cont(\de)+1=\cont(\ga)-e+1=\cont(\ga_e)$.
This implies that $\overline{\ga}$ must be the last element of $P(\bla)$ instead of $\ga_e$, which is a contradiction.

\item Suppose that $\ga_e$ is not removable (cf Example \ref{exalemisopsi1}, 2.). 
This means that there exists a part of  size $\al$ below the part whose rightmost node is $\ga_e$.
Note that the rightmost node $\tga$ of this part has content $\cont(\tga)=\cont(\ga_e)-1$.
By the same cylindricity argument used in 1.(a), this part of size $\al$ spreads in all components of $\bla$.
This means that there exists a part $\al$ in $\la^{c+1}$ (if $c<l$) or in $\la^1$ (if $c=l$) with rightmost node $\de$ verifying 
$\cont(\de)=\cont(\ga)$ (if $c<l$) and $\cont(\de)=\cont(\ga)-e$ (if $c=l$). In particular, $\de^+$ is an $i$-node.
Moreover, if $\de^+$ is addable, unless, of course, it is in the component $\la^{c_e}$. This can be seen using the exact same argument as in 1.(a).

\end{enumerate}

\item Assume now that $\ga$ is the last element of $P(\bla)$.
First of all, note that if $l>1$, one can never have $c=l$. Indeed, in this case $\ga^+$ would not be an addable node.

\begin{enumerate}

\item Suppose that $s_{c+1} > s_c$ (cf Example \ref{exalemisopsi1}, 3.). 
Then, using Proposition \ref{proppp}, we can claim that there exists a part of size $\al$ in the component $\la^{c+1}$.
Denote $\tga$ its rightmost node. By Lemma \ref{lempp1}, $\cont(\tga)=\cont(\ga)+1=\cont(\ga^+)$, and $\tga$ is an $i$-node.
Now, if $\tga$ is removable, then $\tga$ and $\ga^+$ yield an occurence of $RA$ which contradicts the fact that $\ga^+$ is the good $i$-node of type $A$.
Hence, there is necessarily a part $\al$ below the part whose rightmost node is $\tga$. Denote by $\de$ its rightmost node,
so that $\de^+$ is an $i$-node.

\item Suppose now that $s_{c+1}=s_c$ (cf Example \ref{exalemisopsi1}, 4.). Consider the previous node in $P(\bla)$, denote it by $\ga_{e-1}$.
By Lemma \ref{lempp1} again, $\cont(\ga_{e-1})=\cont(\ga)+1=\cont(\ga^+)$, so that $\ga_{e-1}$ is an $i$-node.
Besides, since $\ga^+$ is addable, $\ga_{e-1}$ is removable unless there is a part $\al$ below the part whose rightmost node is $\ga_{e-1}$.
But if $\ga_{e-1}$ is removable, then there exists an addable $i$-node $\tga$ in $\la^{\tc}$ with $\tc \in \llb c+1 , \dots, c_1-1 \rrb$, 
where $c_1$ is the component of $\bla$ which contains the first node of $P(\bla)$
(otherwise $\ga^+$ and $\ga_{e-1}$ yield and occurence $RA$ and $\ga^+$ cannot be the good addable $i$-node).
Then, by the cylindricity argument again, there is a part $\al$ in each component $\la^d$ with $d\in\llb c+1,\dots, \tc\rrb$, 
whose rightmost node has content $\cont(\ga)$.
In particular, this is true for $d=c+1$.
We denote $\de$ the one located in the component $\la^{c+1}$.

Now if $\ga_{e-1}$ is not removable, i.e. if there exists a part $\al$ below the part whose rightmost node is $\ga_{e-1}$,
then again the cylindricity implies that this parts spreads to all components $\la^d$ with $d \in \llb c+1 , \dots, c_{e-1}-1 \rrb$, 
where $c_{e-1}$ is the component of $\bla$ which contains $\ga_{e-1}$, and have the same contents.
Again, we denote $\de$ the rightmost node of the part $\al$ of $\la^{c+1}$ which has content $\cont(\de)=\cont(\ga)$.

\end{enumerate}

\item Assume finally that $\ga$ is neither the first nor the last node of $P(\bla)$. Again, $c$ cannot be equal to $l$ because then it would be the first node of $P(\bla)$.

\begin{enumerate}

\item Suppose that $s_{c+1} > s_c$ (cf Example \ref{exalemisopsi1}, 5.). We can use the same arguments as in 2.(a), and define $\de$ in the exact same way.

\item Suppose  that $s_c=s_{c+1}$ (cf Example \ref{exalemisopsi1}, 6	.). 
Then since $\ga^+$ is not the first node of $P(\bla)$, there exists a part $\al$ in a component $\la^{\tc}$
with $\tc>c$ whose rightmost node $\tga$ has content $\cont(\tga)=\cont(\ga)+1$.
If $\tc=c+1$, then we can use the previous case 3.(a).
If $\tc>c+1$, then:

\begin{itemize}
 \item if $\tga$ is removable, then there exists a part $\al$ in a component $\bar{c}$ with $c<\bar{c}<\tc$ whose rightmost node has content $\cont(\ga)$, 
otherwise $\tga$ produces, together with $\ga^+$, an occurence $RA$, whence the usual contradiction.
By cylindricity, such a part $\al$ also exists in the $(c+1)$-th component of $\bla$. We denote $\de$ its rightmost node.

\item if $\tga$ is not removable, then there exists a part $\al$ below the part whose rightmost node is $\tga$, with rightmost content equal to $\cont(\ga)$.
Again, by cylindricity, it also exists in the $(c+1)$-th component of $\bla$, and we denote $\de$ its rightmost node.

\end{itemize}

\end{enumerate}

\end{enumerate}

It is obvious, but important to notice, that $\de$ is not a node of $P(\bla)$.
Hence, there is a node $\De$ of $\psi(\bla)$ which is canonically associated to $\de$.
By Proposition \ref{propcontents}, $\cont_{\psi(\bla)}(\De)=\cont_{\bla}(\de)$ and in fact $\De=(a,\al,c+1)$ if $c<l$ and $\De=(a,\al,1)$ if $c=l$.
In particular $\De^+$ is an $i$-node.
Moreover, it is addable in $\psi(\bla)$ since $\de^+$ is either addable in $\bla$, 
or the rightmost node of a part above which sit parts that are deleted after applying $\psi$.

In fact, $\De^+$ is the good addable $i$-node of $\psi(\bla)$.
Indeed, consider the $i$-word for $\psi(\bla)$. Denote it by $w_i^\psi$, and denote $w_i$ the $i$-word for $\bla$. 
By construction of $\psi(\bla)$, the subword of $w_i$ corresponding to the rightmost nodes
of the parts that are either greater than $\al$ (respectively smaller than or equal to $\al$ but not in $P(\bla)$)
is also a subword of $w_i^\psi$.
The only differences that are likely to appear are the following:

\begin{itemize}
 \item  The letters $R$ and $A$ that correspond to nodes in $P(\bla)$ vanish. Note that we have assumed that there is always such a letter $A$ in $w_i$
(since $\ga$ is in $P(\bla)$).

\item The parts of size $\al$ that are below a part whose rightmost node is in $P(\bla)$ give a new letter $A$ in $\psi(\bla)$.
 \end{itemize}

Now by construction of $\de$:
\begin{enumerate}
 \item If $\de^+$ gives a letter $A$ in $w_i$, then it is adjacent to the letter $A$ encoding $\ga^+$, to its left.
Hence, since the $A$ encoding $\ga^+$ is no longer in $w_i^\psi$, the letter $A$ corresponding to $\De^+$ in $w_i^\psi$
plays the same role as the one corresponding to $\ga^+$ in $\bla$: it is the rightmost $A$ in the reduced $i$-word of $\psi(\bla)$.
In other terms, $\De^+$ is the good addable $i$-node of $\psi(\bla)$.
 \item If $\de^+$ does not give a letter $A$ in $w_i$, that is if there is an element of $P(\bla)$ just above $\de$,
then it is clear that, again, the $A$ encoding $\De^+$ in $w_i^\psi$ plays the same role as the $A$ encoding $\ga^+$ in $w_i$,
and that $\De^+$ is the good addable $i$-node of $\psi(\bla)$.
\end{enumerate}

\begin{flushright}
$\square$
\end{flushright}

\begin{exa} \label{exalemisopsi1} \hfil
 \begin{enumerate}
  \item $\bla=(4.2,2^2,5.2)$, $\bs=(2,3,4)$, $e=3$ and $i=2$.
\item $\bla=(6^2.2.1^3,4.2^3.1^2,6.2^2.1^4)$, $\bs=(5,6,8)$, $e=4$ and $i=1$.
\item $\bla=(2.1,1^3,1)$, $\bs=(3,4,5)$, $e=4$ and $i=3$.
\item $\bla=(1,1,1^3)$, $\bs=(4,4,7)$, $e=4$ and $i=1$.
\item $\bla=(3.2.1^2,4.2.1,2^3)$, $\bs=(2,3,4)$, $e=4$ and $i=0$.
\item $\bla=(2^2,3.2,2)$, $\bs=(3,4,4)$, $e=3$ and $i=0$.
 \end{enumerate}

\end{exa}

\medskip

\textit{Proof of Lemma \ref{lemisopsi2}.}
First, note that $\Ga$ is nothing but the node of  $\psi(\bla)$ canonically associated to $\ga$ in the definition of $\psi$.
Besides, by definition of $\psi(\bla)$, together with Proposition \ref{propcontents}, we are ensured that $\Ga^+$ is in fact an addable $i$-node of $\psi(\bla)$.
Moreover, $w_i^\psi$ is likely to contain new letters $A$, namely the one corresponding to nodes $\eta^+$ with $\eta=(b,\al,d)\notin P(\bla)$ 
lying just below a node of $P(\bla)$. Denote by H the node of $\psi(\bla)$ corresponding to $\eta$.
Note that $\cont(\eta)\neq\max_{\tga\in P(\bla)}(\cont(\tga))$.
Indeed, since there is a node of $P(\bla)$ just above $\eta$, it is not the first node of $P(\bla)$
and hence it does not have maximal content.
Since $\eta\notin P(\bla)$, there exists a node $\tga=(\ta,\al,\tc)\in P(\bla)$ such that $\cont(\tga)=\cont(\eta)$, $\tc<d$, and $\tga^+$ is addable.
Hence H plays the same role in $\psi(\bla)$ as $\tga$ in $\bla$. 
In particular, since the good addable $i$-node of $\bla$ is in a part of size smaller than $\al$, there is necessarily a letter $R$, encoding a node $\rho$, 
that simplifies with the letter $A$ encoding $\tga^+$.
Now:
\begin{itemize}
\item If $\rho$ is the rightmost node of a part of size different than $\al$, then it is obviously not in $P(\bla)$.
 \item If $\rho$ is the rightmost node of a part of size $\al$ and $\cont(\rho)=\cont(\eta^+)$, then it is not in $P(\bla)$ either.
Indeed, there is a node of $P(\bla)$ just above $\eta$ whose content is $\cont(\eta^+)$, which is therefore not of type $R$, and hence different from $\rho$,
and since all nodes of $P(\bla)$ have different contents, $\rho$ is not in $P(\bla)$.
\item If $\rho$ is the rightmost node of a part of size $\al$ and $\cont(\rho)<\cont(\eta^+)$,
then it is not in $P(\bla)$ either, because the contents of the nodes of $P(\bla)$ are consecutive (cf. Lemma \ref{lempp1}), and
because $\rho$ does not have maximal content.
\end{itemize}
Therefore, the letter $R$ encoding $\rho$ is also present in $w_i^\psi$, and simplifies with the $A$ encoding H.
This implies that $\Ga^+$ is the good addable $i$-node of $\psi(\bla)$.

\begin{flushright}
$\square$
\end{flushright}

\medskip

\textit{Proof of Lemma \ref{lemisopsi3}.}
As in Lemma \ref{lemisopsi2}, $\Ga$ is the node of $\psi(\bla)$ canonically associated to $\ga$.
Consider the letter $A$ encoding $\ga^+$. It is the rightmost letter $A$ in $w_i$ and does not simplify.
By Proposition \ref{propcontents}, $\Ga^+$ is also encoded by a letter $A$ which the righmost letter $A$ in $w_i^\psi$.
It remains to show that it does not simplify with any letter $R$ either.
In fact, as noticed in the previous proofs, the deletion of the pseudoperiod, in the construction of $\psi(\bla)$, cannot yield any new letter $R$.
However, some letters $A$ encoding nodes of $P(\bla)$ can vanish. Denote by $\eta^+$ such a node.

Suppose first that $\eta$ is not the first node of $P(\bla)$.
Since $\eta^+$ is addable, there cannot be another element of $P(\bla)$ above $\eta$.
Then denote by $\eta_1$ the node of $P(\bla)$ which has content $\cont(\eta_1)=\cont(\eta)+1$  (i.e. the previous node of $P(\bla)$).
Note that if $\eta$ is in $\la^d$, then $\eta_1$ is in $\la^{d+1}$.
Suppose now that $\eta$ is the first node of $P(\bla)$. Then, similarly, consider the last node of $P(\bla)$ and denote it by $P(\bla)$.
In each case, either:
\begin{itemize}
 \item there is a part $\al$ just below the part whose rightmost node is $\eta_1$, in which case the node just below $\eta_1$
is not in $P(\bla)$ and yields an addable node in $\psi(\bla)$ which plays exactly the same role in $\psi(\bla)$ as $\eta$ in $\bla$;
\item or there is no node just below $\eta_1$, in which case $\eta_1$ is encoded by a letter $R$ which simplifies with
the letter $A$ encoding $\eta^+$.
\end{itemize}

As a consequence, we are ensured that the letter $A$ encoding $\Ga^+$ does not simplify in $w_i^\psi$, and hence $\Ga^+$ is the good addable
$i$-node of $\psi(\bla)$.

\begin{flushright}
$\square$
\end{flushright}

\medskip

\textbf{Acknowledgments:} I thank Ivan Losev for bringing an important proof to my attention,
as well as Peng Shan for some clarifying conversations. 
My gratitude goes to C\'edric Lecouvey and Nicolas Jacon for their availability and commitment.

\bibliographystyle{plain}
% \bibliography{/users/gerber/Documents/Math/PhD/Biblio/biblio}
% \bibliography{C:/DATA/Documents/Maths/PhD/Biblio/biblio}
\bibliography{biblio}

\begin{thebibliography}{10}

\bibitem{Ariki1996}
Susumu Ariki.
\newblock {On the decomposition numbers of the {H}ecke algebra of
  {$G(m,1,n)$}}.
\newblock {\em J. Math. Kyoto Univ.}, 36(4):789--808, 1996.

\bibitem{Ariki2002}
Susumu Ariki.
\newblock {\em Representations of quantum algebras and combinatorics of {Y}oung
  tableaux}, volume~26 of {\em University Lecture Series}.
\newblock American Mathematical Society, Providence, RI, 2002.
\newblock Translated from the 2000 Japanese edition and revised by the author.

\bibitem{Ariki2007}
Susumu Ariki.
\newblock {On the classification of simple modules for cyclotomic Hecke
  algebras of type G(m,1,n) and Kleshchev multipartitions}.
\newblock {\em Osaka J. Math.}, 38:827--837, 2007.

\bibitem{BrundanKleshchev2001}
Jonathan Brundan and Alexander Kleshchev.
\newblock {Representation Theory of Symmetric Groups and Their Double Covers}.
\newblock In {\em Groups, Combinatorics \& Geometry}, pages 31--53, 2001.

\bibitem{FLOTW1999}
Omar Foda, Bernard Leclerc, Masato Okado, Jean-Yves Thibon, and Trevor Welsh.
\newblock {Branching functions of $A_{n-1}^{(1)}$ and Jantzen-Seitz problem for
  Ariki-Koike algebras}.
\newblock {\em Adv. Math.}, 141:322--365, 1999.

\bibitem{Fulton1997}
William Fulton.
\newblock {\em {Young Tableaux}}.
\newblock Cambridge University Press, 1997.

\bibitem{GeckJacon2011}
Meinolf Geck and Nicolas Jacon.
\newblock {\em {Representations of Hecke Algebras at Roots of Unity}}.
\newblock Springer, 2011.

\bibitem{GerberHissJacon2014}
Thomas Gerber, Gerhard Hiss, and Nicolas Jacon.
\newblock {Harish-Chandra series in finite unitary groups and crystal graphs}.
\newblock 2014.
\newblock {\tt arXiv:1408.1210}.

\bibitem{GordonLosev2014}
Iain Gordon and Ivan Losev.
\newblock {On category $\cO$ for cyclotomic rational Cherednik algebras}.
\newblock {\em J. Eur. Math. Soc.}, 16:1017--1079, 2014.

\bibitem{Grojnowski1999}
Ian Grojnowski.
\newblock {Affine $\mathfrak{sl}_p$ controls the representation theory of the
  symmetric group and related Hecke algebras}.
\newblock 1999.

\bibitem{HongKang2002}
Jin Hong and Seok-Jin Kang.
\newblock {\em {Introduction to Quantum Groups and Crystal Bases}}.
\newblock American Mathematical Society, 2002.

\bibitem{Jacon2007}
Nicolas Jacon.
\newblock {Crystal graphs of irreducible $U_v(\hat{sl}_e)$-modules of level two
  and Uglov bipartitions}.
\newblock {\em Journal of Algebraic Combinatorics}, 27:143--162, 2007.

\bibitem{JaconLecouvey2010}
Nicolas Jacon and C\'edric Lecouvey.
\newblock {Crystal isomorphisms for irreducible highest weight
  $U_{v}(\widehat{\mathfrak{sl}_{e}})$-modules of higher level}.
\newblock {\em Algebras and Representation Theory}, 13:467--489, 2010.

\bibitem{JaconLecouvey2012}
Nicolas Jacon and C\'edric Lecouvey.
\newblock {A combinatorial decomposition of higher level Fock spaces}.
\newblock {\em Osaka Journal of Mathematics}, 2012.

\bibitem{JamesKerber1984}
Gordon James and Adalbert Kerber.
\newblock {\em {The Representation theory of the Symmetric Group}}.
\newblock Cambridge University Press, 1984.

\bibitem{JMMO1991}
Michio Jimbo, Kailash~C. Misra, Tetsuji Miwa, and Masato Okado.
\newblock Combinatorics of representations of {$U_q(\widehat{sl(n)})$} at
  {$q=0$}.
\newblock {\em Comm. Math. Phys.}, 136(3):543--566, 1991.

\bibitem{Kac1984}
Victor~G. Kac.
\newblock {\em {Infinite Dimensional Lie Algebras}}.
\newblock Birkh\"auser, 1984.

\bibitem{Kashiwara1991}
Masaki Kashiwara.
\newblock {On crystal bases of the $q$-analogue of universal enveloping
  algebras}.
\newblock {\em Duke Math. J.}, 63:465--516, 1991.

\bibitem{KashiwaraNakashima1994}
Masaki Kashiwara and Toshiki Nakashima.
\newblock {Crystal graphs for representations of the $q$-analogue of classical
  Lie algebras}.
\newblock {\em Journal of Algebra}, 165:295--345, 1994.

\bibitem{LLT1996}
Alain Lascoux, Bernard Leclerc, and Jean-Yves Thibon.
\newblock {Hecke algebras at roots of unity and crystal bases of quantum affine
  algebras}.
\newblock {\em Comm. Math. Phys.}, 181:205--263, 1996.

\bibitem{LascouxLeclercThibon1995}
Alain Lascoux, Bernard Leclerc, and Jean-Yves Thibon.
\newblock {Crystal Graphs and $q$-Analogues of Weight Multiplicities for the
  Root System $A_n$}.
\newblock {\em Lett. Math. Phys.}, 35:359--374, 2009.

\bibitem{Leclerc2008}
Bernard Leclerc.
\newblock {Fock spaces representations of $U_q(\hat{sl_n})$}.
\newblock In {\em Geometric Methods in Representation Theory}, 2008.
\newblock Notes from the Summer School held at Institut Fourier, Grenoble.

\bibitem{Lecouvey2007}
C\'edric Lecouvey.
\newblock {Combinatorics of Crystal Graphs for the Root Systems of Types $A_n$,
  $B_n$, $C_n$, $D_n$ and $G_2$}.
\newblock {\em MSJ Memoirs}, 17 (Combinatorial Aspects of Integrable
  Systems):11--41, 2007.
\newblock Papers from the workshop held at RIMS, Kyoto University, July 26-30,
  2004.

\bibitem{Losev2013}
Ivan Losev.
\newblock {Highest weight $\mathfrak{sl}_2$-categorifications I: crystals}.
\newblock {\em Math. Z.}, 274:1231--1247, 2013.

\bibitem{Lothaire2002}
M.~Lothaire.
\newblock {\em {Algebraic combinatorics on words}}, volume~90 of {\em
  Encyclopedia of Mathematics and its Applications}.
\newblock Cambridge University Press, 2002.

\bibitem{Shan2008}
Peng Shan.
\newblock {Crystals of Fock spaces and cyclotomic rational double affine Hecke
  algebras}.
\newblock {\em Ann. Sci. \'Ec. Norm. Sup\'er.}, 44:147--182, 2011.

\bibitem{ShanVasserot2012}
Peng Shan and Eric Vasserot.
\newblock {Heisenberg algebras and rational double affine Hecke algebras}.
\newblock {\em J. Amer. Math. Soc.}, 25:959--1031, 2012.

\bibitem{Shimozono2001}
Mark Shimozono.
\newblock {A cyclage poset structure for Littlewood-Richardson tableaux}.
\newblock {\em European J. Combin.}, 22:365--393, 2001.

\bibitem{Tingley2008}
Peter Tingley.
\newblock {Three combinatorial models for $\widehat{\fs\fl_n}$ crystals, with
  applications to cylindric plane partitions}.
\newblock {\em Int. Math. Res. Notices}, 2:Art. ID rnm143, 40 pp., 2008.

\bibitem{Uglov1999}
Denis Uglov.
\newblock {Canonical bases of higher-level q-deformed Fock spaces and
  Kazhdan-Lusztig polynomials}.
\newblock {\em Progr. Math.}, 191:249--299, 1999.

\bibitem{VaragnoloVasserot1999}
Michela Varagnolo and \'Eric Vasserot.
\newblock {On the decomposition matrices of the quantized Schur algebra}.
\newblock {\em Duke Math. J.}, 100:267--297, 1999.

\bibitem{Yvonne2005}
Xavier Yvonne.
\newblock {\em {Bases canoniques d'espaces de Fock de niveau sup\'erieur}}.
\newblock PhD thesis, Universit\'e de Caen, 2005.

\end{thebibliography}

\textbf{Mathematics Subject Classification :} 05E10, 17B37, 20C08, 68R15.

\end{document}